\documentclass[12pt,xcolor=dvipsnames]{amsart}
\usepackage{mathtools}
\usepackage[english]{babel}
\usepackage{epsfig}
\usepackage{mathptmx}
\usepackage{amsmath,amsfonts,amssymb}
\usepackage{mathrsfs}
\usepackage{graphicx}
\usepackage{latexsym}
\usepackage{verbatim}
\usepackage{enumerate}
\usepackage[svgnames]{xcolor}
\usepackage[normalem]{ulem}
\usepackage{cancel}
\usepackage{tabularx}
\usepackage{caption}
\usepackage{tikz-cd}
\tikzcdset{every label/.append style = {font = \normalsize}}
\usepackage[backref=page]{hyperref} 
\usepackage{wasysym}
\usepackage{soul}

\numberwithin{equation}{section}

\def\eps{\varepsilon}
\newcommand{\D}{\mathbb D}
\newcommand{\R}{\mathbb R}
\newcommand{\Z}{\mathbb Z}
\newcommand{\C}{\mathbb C}
\newcommand{\N}{\mathbb N}

\def\spa{\mathop{{\rm span}}}

\newenvironment{mylist}{\begin{list}{}%
		{\labelwidth=2em\leftmargin=\labelwidth\itemsep=.4ex plus.1ex minus.1ex\topsep=.7ex plus.3ex minus.2ex}%
		\let\itm=\item\def\item[##1]{\itm[{\rm ##1}]}}{\end{list}}

\def\Aut{{\sf Aut}}

\newtheoremstyle{break}
{8pt}{8pt}%
{\itshape}{}%
{\bfseries}{}%
{\newline}{}%
\theoremstyle{break}
\theoremstyle{theorem}

\newcommand{\at}[2][\big]{#1\vert_{#2}}
\newcommand{\argument}{\,\cdot\,}
\newcommand{\cc}[1]{\overline{{#1}}}
\DeclarePairedDelimiter{\abs}{\lvert}{\rvert}
\DeclarePairedDelimiter{\norm}{\lVert}{\rVert}
\newcommand{\del}{\mathop{}\!\partial}
\newcommand{\delbar}{\mathop{}\!\cc{\partial}}
\newcommand{\SymD}{\operatorname{D}_{\operatorname{sym}}}

\allowdisplaybreaks

\theoremstyle{break}

\newtheorem{theorem}{Theorem}[section]
\newtheorem*{theorem*}{Theorem}
\newtheorem{lemma}[theorem]{Lemma}
\newtheorem*{lemma*}{Lemma}
\newtheorem{proposition}[theorem]{Proposition}
\newtheorem{corollary}[theorem]{Corollary}

\theoremstyle{break}
\newtheorem{definition}[theorem]{Definition}

\theoremstyle{remark}
\newtheorem{remark}{Remark}

\numberwithin{equation}{section}
\makeatletter
\def\blfootnote{\xdef\@thefnmark{}\@footnotetext}
\makeatother

\author[M. Heins]{Michael Heins}
\address{M. Heins: Department of Mathematics, University of W\"urzburg, Emil Fischer Strasse 40, 97074, W\"urzburg, Germany.} \email{michael.heins@mathematik.uni-wuerzburg.de}

\author[A.~Moucha]{Annika Moucha$^\S$}
\address{A. Moucha: Department of Mathematics, University of W\"urzburg, Emil Fischer Strasse 40, 97074, W\"urzburg, Germany.} \email{annika.moucha@uni-wuerzburg.de}

\author[O. Roth]{Oliver Roth}
\address{O. Roth: Department of Mathematics, University of W\"urzburg, Emil Fischer Strasse 40, 97074, W\"urzburg, Germany.} \email{roth@mathematik.uni-wuerzburg.de}

\author[T. Sugawa]{Toshiyuki Sugawa$^\dagger$}
\address{T. Sugawa: Graduate School of Information Sciences, Tohoku University, Aoba-Ku, Sendai 980-8579, Japan.} \email{sugawa@math.is.tohoku.ac.jp}

\title[Peschl--Minda derivatives and convergent Wick star product]{Peschl--Minda derivatives and convergent Wick star products\\[2mm] on the disk, the sphere and beyond}

\date{\today}
\thanks{$^\S\,$Partially supported by the Alexander von Humboldt Stiftung. $\null^\dagger\,$Supported in part by JSPS KAKENHI Grant Number JP17H02847}

\dedicatory{Dedicated to David Minda}

\begin{document}
	\maketitle
	\begin{abstract}
 We introduce and study invariant differential operators acting on the space $\mathcal{H}(\Omega)$ of holomorphic functions on the complement ${\Omega=\{(z,w) \in \hat{\mathbb{C}}^2 \, : \, z\cdot w \not=1\}}$ of the ``complexified unit circle'' $\{(z,w) \in \hat{\mathbb{C}}^2 \, : \, z\cdot w =1\}$.
We obtain recursion identities, describe the behaviour under  change of coordinates and find the generators of the corresponding operator algebra.  We illustrate how this provides a unified framework for investigating conformally invariant differential
 operators  on the unit disk $\mathbb{D}$ and the Riemann sphere~$\hat{\mathbb{C}}$, which have been studied by Peschl, Aharonov, Minda and many others,  within their conjecturally natural habitat.
We  apply the machinery to a problem in  deformation quantization by deriving  explicit formulas for the canonical Wick--type star products on $\Omega$, the unit disk $\mathbb{D}$ and the Riemann sphere~$\hat{\mathbb{C}}$ in terms of such invariant differential operators.
 These formulas are given in  form of  factorial series which depend holomorphically on  a complex deformation parameter $\hbar$ and lead to asymptotic expansions of the star products in powers of $\hbar$.
    \end{abstract}

   \section{Introduction}
    Invariant differential operators for holomorphic functions  have been of continual interest
    in complex analysis for a long time. While the principal idea can be traced back at least to the classical work of Schwarz \cite{Schwarz1873}, the systematic study of conformally invariant differential operators for holomorphic functions of one complex variable has been initiated by Peschl \cite{Peschl1955} for the case of the hyperbolic and spherical metric. Important applications of those invariant operators have been given by e.g. Aharonov \cite{Ahar69}, Harmelin \cite{Harmelin1982}, Schippers \cite{Schippers2003}, and Kim and Sugawa \cite{KS09}. 
        The theory has been extended by Nakahara \cite{Nakahara2003}, Minda \cite{Minda}, Schippers \cite{Schippers1999,Schippers2003,Schippers2003b,Schippers2007},  and Kim and Sugawa \cite{KS07diff} to the case of general conformal metrics on Riemann surfaces and also to include smooth but not necessarily holomorphic functions by using Wirtinger $\partial$--derivatives instead of ordinary complex derivatives, which require holomorphicity. We follow \cite{KS09} and call these derivatives \textit{Peschl--Minda derivatives}.

    \medskip

    In this paper we introduce and study  invariant derivatives for smooth 
    functions $f : U \to \C$ of \textit{two independent} complex variables $(z,w) \in U$ in such a way that
    restriction to the ``diagonal'' $\{(z,\overline{z})\in U\}$ resp.~the ``rotated diagonal''  $\{(z,-\overline{z})\in U \}$ recaptures the aforementioned one--variable Peschl--Minda derivatives for the unit disk $\D:=\{z \in \C \, : \, |z|<1\}$ resp.~the Riemann sphere $\hat{\C}:=\C \cup \{\infty\}$. In fact, some of the resulting identities in the two--variable setting do have well--established counterparts for  one--variable Peschl--Minda derivatives. This is hardly surprising in view of the well--known ``identity principle'' (\cite[p.~18]{Range}) that a holomorphic function $f : U \to \C$ on a domain $U \subseteq \C^2$ containing a point of the form $(z,\cc{z})$ or $(z,-\cc{z})$ is already determined by one of its traces $\{f(z,\overline{z}) \, : \, (z,\overline{z}) \in U\}$ and
    $\{f(z,-\overline{z}) \, : \, (z,-\overline{z}) \in U\}$.
    However, it turns out that identifying  the set
   \begin{equation}
        \label{eq:Omega}
        \Omega
        \coloneqq
        \hat{\C}^2
        \setminus
        \{(z,w)\in\hat{\C}^2 \, \colon \, z\cdot w = 1\}\footnotemark  \, ,
      \end{equation}\footnotetext{Here, we extend the arithmetic in $\C$ in the usual manner by $z \cdot \infty = \infty = \infty \cdot z$ for $z \in \hat{\C} \setminus \{0\}$ and $0 \cdot \infty = 1 = \infty \cdot 0$.
        We think of $\Omega$ as  the complement of the complexified unit circle $\{(z,w) \in \hat{\C}^2 \, : \, z \cdot w=1\}$.}as the    \textit{maximal} subdomain of $\hat{\C}^2$ such that the two--variable   Peschl--Minda derivatives are defined for every holomorphic function $f : \Omega \to \C$  leads to a coherent viewpoint  which  connects the Peschl--Minda derivatives with the spectral theory of the invariant Laplacians on the unit disk $\D$ and the Riemann sphere $\hat{\C}$ (see Helgason \cite{Helgason70} and Rudin \cite{Rudin84}) as well as with recent work on strict deformation quantization of $\D$ and $\hat{\C}$ (see \cite{CahenGuttRawnsley1994}, \cite{KrausRothSchoetzWaldmann2019}, \cite{SchmittSchoetz2022}). It is the purpose of this paper and its companions \cite{HeinsMouchaRoth3,HeinsMouchaRoth2,KrausRothSchleissinger,Annika} to develop this point of view and to discuss its ramifications. The focus of the present paper is on Peschl--Minda derivatives and their applications to strict deformation quantization.

            \medskip
  We briefly indicate this application.
  The Fréchet space $\mathcal{H}(\Omega)$ of all holomorphic functions $f \colon \Omega \longrightarrow \C$, equipped with its natural topology of locally uniform convergence, plays a special, but in some sense also peculiar key role in recent work \cite{BeiserWaldmann2014, EspositoSchmittWaldmann2019, KrausRothSchoetzWaldmann2019, SchmittSchoetz2022} on strict deformation quantization of the unit disk and the Riemann sphere. In \cite{SchmittSchoetz2022} the authors
  provide a partial explanation by considering so--called Wick rotations. For our purposes, it suffices to think of~a strict deformation of a Fréchet algebra $(\mathcal{A}(U),+,\cdot)$ of smooth functions $f : U \to \C$ as a family of continuous (typically non--commutative) multiplications $\star_\hbar$ on $\mathcal{A}(U)$ depending on a complex parameter $\hbar$ with $\star_0$ being the pointwise product $\cdot$. We refer to the multiplications $\star_\hbar$ as \emph{star~products}. 
  For a comprehensive survey of strict deformation quantization we refer to \cite{Waldmann2019}, and for relations to Bergman spaces to \cite{CahenGuttRawnsley1994}.

    \medskip

    The results about Peschl--Minda operators on $\mathcal{H}(\Omega)$, which we establish in this paper, provide a conceptual explanation of the role the domain $\Omega$ is playing in strict deformation quantization of the unit disk $\D$ and the Riemann sphere $\hat{\C}$.  It turns out that the so--called Wick star products on $\D$ and  $\hat{\C}$ may be understood as pullbacks of a Wick star product $\star_\hbar$ on the ``ambient'' Fréchet algebra $\mathcal{H}(\Omega)$. While this fact is already implicit in the original construction \cite{KrausRothSchoetzWaldmann2019}, see also \cite{SchmittSchoetz2022}, the novelty is that the star product $f \star_\hbar g$ on $\mathcal{H}(\Omega)$ may be expressed in terms of a factorial series with respect to the formal parameter $\hbar$ (see  \cite{SchmittSchoetz2022}) \textit{and}  with Peschl--Minda derivatives of $f$ and $g$ as coefficients. This has two consequences: first, since we prove that Peschl--Minda derivatives do have strong invariance properties, the star product on $\mathcal{H}(\Omega)$ is fully invariant under a distinguished subgroup $\mathcal{M}$ of the group $\Aut(\Omega)$ of all biholomorphic self maps of $\Omega$. This subgroup $\mathcal{M}$ is  induced in a natural way by the group of all M\"obius transformations acting on~$\hat{\C}$. Therefore, roughly speaking, the  star product on $\mathcal{H}(\Omega)$ is fully M\"obius--invariant. From this fact the invariance of the star products on $\D$ and on $\hat{\C}$ w.r.t.~their intrinsic automorphism groups follows at once. The second consequence, which we deduce from  the representation of the star product in terms of a factorial series, is  an asymptotic expansion of the star product $f \star_\hbar g$ in powers of $\hbar$ as $\hbar \to 0$, which, loosely speaking, models the passage from quantum theory to classical mechanics.  Previously, only the first two approximate terms have been identified (see \cite[Theorem 4.5]{KrausRothSchoetzWaldmann2019}).

  \medskip

    The paper is structured as follows. In Section \ref{sec:Automorphisms} we introduce and study the subgroup $\mathcal{M}$ of $\Aut(\Omega)$ which naturally corresponds to the  group $\Aut(\hat{\C})$ of all M\"obius transformations.
The M\"obius--type automorphisms in $\mathcal{M}$ are essential for the definition of the Peschl--Minda differential operators in Section \ref{sec:PM}, where we also study in detail their transformation behaviour under change of coordinates, and clarify in Theorem \ref{thm:PMwithLaplace} the structure of the operator algebra they generate. The building blocks of this operator algebra, which we call pure Peschl--Minda derivatives, are studied in Section \ref{sec:PurePM}. We relate them in two different ways to ordinary euclidean derivatives, see Theorem~\ref{thm:PMPureEuclidean} and Theorem~\ref{thm:PMLinearisation}. These results also provide convenient tools for computing pure Peschl--Minda operators, and as  an application we determine their kernels. In Section \ref{sec:PMClassical} we discuss the relation between the two--variable Peschl--Minda derivatives and their classical one--variable analogues. In Section \ref{sec:StarProduct} we apply the previous results to construct~a Wick--type star product on $\mathcal{H}(\Omega)$, establish its invariance under the full M\"obius--type group~$\mathcal{M}$, and deduce an asymptotic power series expansion for the star product, see Theorem \ref{thm:StarProductAsymptotics}. In the final section, we indicate how to extend the star product to a larger class of holomorphic functions than $\mathcal{H}(\Omega)$ if one allows continuous module structures rather than Fr\'echet algebras.

	\section{M\"obius--type automorphisms of $\Omega$}
	\label{sec:Automorphisms}

	We denote the punctured plane by $\C^* \coloneqq \C \setminus \{0\}$, the Riemann sphere by $\hat{\C}$,  and the punctured sphere by $\hat{\C}^* \coloneqq \hat{\C} \setminus \{0\}$. Moreover, we write $\N \coloneqq \{1, 2, \ldots\}$ for the set of positive integers, $\N_0 \coloneqq \N \cup \{0\}$, and $\Z$ for the set of all integers. 
	For an open subset $U$ of $\hat{\C}$ or $\hat{\C}^2$, we denote the set of all functions $f \colon U \to \C$ which are infinitely  (real) differentiable by $C^{\infty}(U)$ and call such functions smooth. Finally, $\mathcal{H}(U)$ denotes the set of all holomorphic functions $f: U\to\C$, and $\Aut(U)$ the set of all biholomorphic mappings or automorphisms $T : U \to U$. We often use the fact that $\Aut(U$) is a group w.r.t.~composition.

	\medskip

	Our object of interest is the subdomain $\Omega \subseteq \hat{\C}^2$ defined in \eqref{eq:Omega}. See Figure~\ref{fig:Omega} for a schematic picture of $\Omega$. The purpose of this short section is to collect some basic properties of $\Omega$ and its automorphisms. Clearly, the \textit{flip}
	\begin{equation*}
		\mathcal{F} : \Omega \longrightarrow \Omega \, , \qquad \mathcal{F}(z,w)
		\coloneqq
		\Big(
		\frac{1}{w}, \frac{1}{z}
		\Big)
		\, , \qquad
		(z,w) \in \Omega
	\end{equation*}
	belongs to $\Aut(\Omega)$, and the group $\Aut(\hat{\C})$ of all M\"obius transformations acts on $\Omega$ in the sense that for every $\psi \in \Aut(\hat{\C})$
			the mapping
			\begin{equation}
				\label{eq:MoebiusAction}
				T(z,w)
				\coloneqq
				\big(
				\psi(z), 1/\psi(1/w)
				\big)
			\end{equation}
			induces an automorphism $T$ of $\Omega$.
			\begin{figure}[t]
				\centering
				\includegraphics[width = \textwidth]{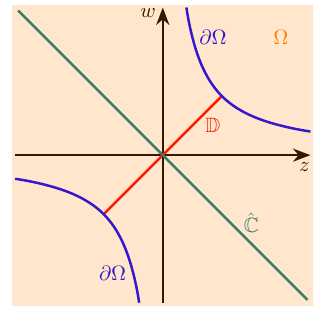}
				\caption{Schematic picture of the domain $\Omega$}
				\label{fig:Omega}
			\end{figure}
			In the sequel, we restrict ourselves to the special type (\ref{eq:MoebiusAction}) of automorphisms of $\Omega$ and the flip $\mathcal{F}$.  Accordingly, we write
			\begin{eqnarray*}
				\mathcal{M}^+
				&\coloneqq&
				\left\{
				T \in \Aut(\Omega)
				\colon
				\, T(z,w)
				=
				\big( \psi(z), 1/\psi(1/w) \big)
				\text{ for some } \psi \in \Aut(\hat{\C})
				\right\} \, ,\\
				\mathcal{M}^- &\coloneqq& \mathcal{M}^+ \circ \mathcal{F}    \quad \text{ and } \\
				\mathcal{M}\phantom{^+}
				&\coloneqq&
				\mathcal{M}^+
				\cup \mathcal{M}^-
				\, .
			\end{eqnarray*}
			We refer to the elements of $\mathcal{M}$ as M\"obius--type transformations. By \cite[Theorem~5.2]{HeinsMouchaRoth3} the M\"{o}bius--type transformations are precisely those automorphisms of $\Omega$ which leave the Laplacian of $\Omega$, see (\ref{eq:Laplacian}), invariant.
			The flip $\mathcal{F}$ has the simple, but important property that it is central in $\mathcal{M}$, that is,
			\begin{equation}
				\label{eq:FlipCentral}
				T \circ \mathcal{F}
				=
				\mathcal{F} \circ T
				\qquad
				\text{ for all }
				T \in \mathcal{M} \, .
			\end{equation}
			Clearly, $\mathcal{M}$ is a subgroup of $\Aut(\Omega)$. Moreover,
			$$ T \in \mathcal{M}^+ \quad \Longleftrightarrow \quad T \circ \mathcal{F}=\mathcal{F} \circ T  \in \mathcal{M}^- \, .$$
			\begin{lemma}
				\label{lem:MoebiusGenerators}
				The group $\mathcal{M}$ is generated by the flip $\mathcal{F}$ and the mappings defined by
				\begin{equation}
					\label{eq:MoebiusGenerators}
					\Phi_{z,w}(u,v)
					\coloneqq
					\Big(
					\frac{z+u}{1+wu} ,
					\frac{w+v}{1+zv}
					\Big)
					\quad \text{and} \quad
					\varrho_{\gamma}(u,v)
					\coloneqq
					\Big(
					\gamma u, \frac{v}{\gamma}
					\Big)
					\, , \qquad
					(u,v) \in \Omega
				\end{equation}
				with $(z,w) \in \Omega \cap \C^2$ and $\gamma \in \C^*$. The mappings \eqref{eq:MoebiusGenerators} generate the subgroup $\mathcal{M}^+$. More precisely, for every $T \in \mathcal{M}$ there exist $(z,w) \in \Omega \cap \C^2$ and $\gamma \in \C^*$ such that
				\begin{equation}
					\label{eq:MoebiusGeneratorsExplicit}
					T=\varrho_\gamma \circ \Phi_{z,w} \circ \mathcal{F} \quad \text{ or } \quad T=\varrho_\gamma \circ \Phi_{z,w}
					\, .
				\end{equation}
			\end{lemma}
			\begin{proof}
				By our preliminary considerations and the fact that every $\psi \in \Aut(\hat{\C})$ defines by way of (\ref{eq:MoebiusAction}) a mapping $T \in \mathcal{M}^+$,  it suffices to note that the maps
				\begin{equation}
					\label{eq:MoebiusGeneratorsOneVariable}
					u
					\mapsto
					\frac{1}{u} \, , \qquad
					u
					\mapsto
					\frac{z+u}{1+wu}
					\quad \text{and} \quad
					u
					\mapsto
					\gamma u
					\, , \qquad
					u \in \hat{\C}
				\end{equation}
				with $(z,w) \in \Omega \cap \C^2$ and $\gamma \in \C^*$ generate the Möbius group $\Aut(\hat{\C})$. Finally, recall \eqref{eq:FlipCentral} and note that
				\begin{equation*}
					\Phi_{z,w} \circ \varrho_\gamma
					=
					\varrho_\gamma \circ \Phi_{z/\gamma, w\gamma}
					=
					\varrho_\gamma
					\circ
					\Phi_{\varrho_{1/\gamma}(z,w)} \, .
				\end{equation*}
				Therefore, the ordering of generators in \eqref{eq:MoebiusGeneratorsExplicit} can always be achieved.
			\end{proof}
			The maps $\varrho_\gamma$ are called dilations by $\gamma \in \C^*$ and form an abelian subgroup of $\mathcal{M}^+$. On the other hand, generic compositions $\Phi_{\alpha,\beta} \circ \Phi_{z,w}$ may not be written as $\Phi_{u,v}$. Indeed, straightforward computations yield
			\begin{equation}
				\label{eq:MoebiusTComposition}
				\Phi_{\alpha,\beta}
				\circ
				\Phi_{z,w}
				=
				\varrho_{\frac{1 + \alpha w}{1 + \beta z}}
				\circ
				\Phi_{\Phi_{z,w}(\alpha,\beta)}
				\qquad
				\text{for}
				\quad
				(\alpha, \beta) \in \Omega \cap \C^2
				\, ,
			\end{equation}
			where $z\in\C\setminus\{1/\beta\}$ and $w\in\C\setminus\{1/\alpha\}$. We need to exclude $z=1/\beta$ and $w=1/\alpha$ because otherwise $\Phi_{\Phi_{z,w}(\alpha,\beta)}$ would not be defined.

			\medskip

			Now, fix a point $(z,w) \in \Omega\cap\C^2$. Note that $\Phi_{z,w}$ sends $(0,0)$ to $(z,w)$. Our goal in Section~\ref{sec:PM} is to define ``almost $\mathcal{M}$--invariant derivatives'', using the fact that $\mathcal{M}$ acts transitively on $\Omega$. More precisely, we precompose $f$ with $\Phi_{z,w} \in \mathcal{M}$ and then take the euclidean derivatives of the composition $f \circ \Phi_{z,w}$ at $(0,0)$. However, this only works for $(z,w) \in \Omega \cap \C^2$. To remedy this, it is convenient to employ the following atlas for $\Omega$ formed by the two charts:
			\begin{itemize}
				\item[(i)] \textit{Standard chart}
				\begin{equation}
					\label{eq:ChartStandard}
					\phi_{+}
					\colon
					\Omega \cap \C^2
					\longrightarrow
					\C^2
					\, , \qquad
					\phi_+(u,v)
					\coloneqq
					(u,v) \, .
				\end{equation}
				\item[(ii)] \textit{Flip chart}
				\begin{equation}
					\label{eq:ChartFlip}
					\phi_-
					\colon
					\Omega \cap (\hat{\C}^* \times \hat{\C}^* )
					\longrightarrow
					\C^2
					\, , \qquad
					\phi_-(u,v)
					\coloneqq
					(1/v,1/u) \, .
				\end{equation}
			\end{itemize}

			Note that the flip chart $\phi_-$ is just  the flip $\mathcal{F}\in \mathcal{M}$ restricted to $\Omega \cap (\hat{\C}^* \times \hat{\C}^*)$. Hence, in standard coordinates, a function $f \in C^{\infty}(U)$, $U \subseteq \Omega$, is given by
			\begin{equation*}
				f_+
				\coloneqq
				f \circ \phi_+^{-1}
				=
				f
				\at{U \cap \C^2}
				\in
				C^{\infty}(U \cap \C^2)\, ,
			\end{equation*}
			and in flipped coordinates by
			\begin{equation*}
				f_-
				\coloneqq
				f \circ \phi_-^{-1}
				\in
				C^{\infty}(\phi_{-}(U) \cap \C^2) \, .
			\end{equation*}

			\section{Peschl--Minda Derivatives}
			\label{sec:PM}%
			In this section, we define Peschl--Minda differential operators $D^{m,n}$ of order $(m,n) \in \N_0\times\N_0$ for smooth functions corresponding to the charts \eqref{eq:ChartStandard} and \eqref{eq:ChartFlip} of $\Omega$. We investigate their transformation behaviour with respect to $\mathcal{M}$ and study the structure of the operator algebra they generate. To this end, we need some more notation. By $\partial^m_1 \partial^n_2 f$ we denote the complex (Wirtinger) derivative of some smooth function $f$ of order $m$ w.r.t.~the first variable and of order~$n$ w.r.t. the second variable. If the variables are specified, we also write $\partial_z^n \partial_w^m f(z,w)$.
			\begin{definition}[Peschl--Minda derivatives]
				\label{def:PeschlMindageneral}
				Let $m,n \in \N_0$,  $U$ an open subset of $\Omega$, and $f \in C^{\infty}(U)$.
				\begin{mylist}
					\item[(a)] Let $(z,w) \in U \cap \C^2$. We define
					\begin{equation}
						\label{eq:PMStandard}
						D^{m,n} f(z,w)
						\coloneqq
						\partial^m_1 \partial^n_2
						\big(
						f \circ \Phi_{z,w}
						\big)(0,0)
					\end{equation}
					and call $D^{m,n}f$ the Peschl--Minda derivative of $f$ of order $(m,n)$.
					\item[(b)] Let $(z,w) \in U \cap (\hat{\C}^* \times \hat{\C}^*)$. We define
					\begin{equation}
						\label{eq:PMFlip}
						\widetilde{D}^{m,n} f(z,w)
						\coloneqq
						\partial^m_1 \partial^n_2
						\big(
						f_- \circ \Phi_{\mathcal{F}(z,w)}
						\big)(0,0)
						\, .
					\end{equation}
				\end{mylist}
			\end{definition}

        It is useful to think of the derivatives \eqref{eq:PMStandard} as being associated to the standard chart \eqref{eq:ChartStandard}, while \eqref{eq:PMFlip} corresponds to the flip chart \eqref{eq:ChartFlip}. In fact, we will see in Section~\ref{sec:StarProduct} that \eqref{eq:PMStandard} and \eqref{eq:PMFlip} are coordinate expressions of some natural geometric objects associated to $\Omega$. A computation reveals the first order Peschl--Minda derivatives as
			\begin{equation}
				\label{eq:PMFirstOrder}
				D^{1,0}f(z,w)
				=
				(1-zw)
				\partial_z f(z,w)
				\quad \text{and} \quad
				\widetilde{D}^{1,0} f(z,w)
				=
				\frac{w}{z}
				(1-zw)
				\partial_w f(z,w) \,.
			\end{equation}
			Interchanging $z$ and $w$ in \eqref{eq:PMFirstOrder} yields corresponding formulas for $D^{0,1}f(z,w)$ and $\widetilde{D}^{0,1} f(z,w)$. Computing the higher--order derivatives in terms of $\partial$ and $\overline{\partial}$ just by using the definition quickly becomes quite cumbersome. Studying the operators in an abstract fashion first will lead us to more applicable formulas. The following is immediate from the definition.
			\begin{lemma}
				\label{lem:PMHolomorphic}%
				Let $f \in \mathcal{H}(U)$ for some open $U \subseteq \Omega \cap \C^2$ and $(z,w) \in U$. Then
				\begin{equation*}
					f
					\Big(
					\frac{z+u}{1+wu},
					\frac{w+v}{1+zv}
					\Big)
					=
					\big(
					f \circ \Phi_{z,w}
					\big)(u,v)
					=
					\sum_{m,n=0}^{\infty}
					\frac{D^{m,n} f(z,w)}{m! n!}
					u^m v^n
				\end{equation*}
				holds for $(u,v)$ on some sufficiently small bi--disk around $(0,0)$. In particular, $D^{m,n} f$ is holomorphic on $U$.
			\end{lemma}

			\begin{remark}[Swap symmetry]
				The domain $\Omega$ from \eqref{eq:Omega} is invariant under switching the roles of~$z$ and $w$, i.e. $(z,w)\in\Omega$ if and only if $(w,z)\in\Omega$. By \eqref{eq:PMStandard}, we moreover have
				\begin{equation*}
					D^{m,n} f(w,z)
					=
					D^{n,m} f(z,w)
					\, , \qquad
					(z,w) \in \Omega \cap \C^2
				\end{equation*}
				for all $m,n \in \N_0$ and $f \in C^\infty(U)$. Consequently, proving a result for the operator $D^{m,n}$ immediately implies a corresponding statement for $D^{n,m}$. In the sequel, we will use this symmetry without further comment and only treat one of the cases in each proof.
			\end{remark}

			The next result shows how the operators $D^{m,n}$ and $\tilde{D}^{m,n}$ are related. It will play an important role in what follows.
			\begin{proposition}[$D$ vs. $\tilde{D}$]
				\label{prop:DvsDtilde}
				Let $m,n \in \N_0$, $U$ an open subset of $\Omega$ and $f \in C^{\infty}(U)$. Then
				\begin{equation*}
					\widetilde{D}^{m,n} f(z,w)
					=
					\frac{z^n}{w^n}
					\frac{w^m}{z^m}
					D^{n,m} f(z,w)
				\end{equation*}
				for all $(z,w) \in U \cap (\C^* \times \C^*)$.
			\end{proposition}
\begin{proof}
Unwrapping the definitions yields
				\begin{equation} \label{eq:DvsDtildeProof}
					\big( \mathcal{F} \circ \Phi_{\mathcal{F}(z,w)} \big)
					(u,v)
					=
					\mathcal{F}
					\bigg(
					\frac{1/w + u}{1+u/z},
					\frac{1/z + v}{1+v/w}
					\bigg)
					=
					\bigg(
					\frac{z}{w}
					\frac{w+v}{1+zv},
					\frac{w}{z}
					\frac{z+u}{1+wu}
					\bigg)
				\end{equation}
				for $(z,w) \in \Omega \cap (\C^* \times \C^*)$ and $(u,v) \in \C$.
Fixing $(z,w) \in \Omega \cap (\C^* \times \C^*)$, and setting $h=f\circ\varrho_\gamma$
for $\gamma=z/w$ and $g(u,v)=h(v,u),$ we have the expression
				\begin{equation*}
					\big(
					f_- \circ \Phi_{\mathcal{F}(z,w)}
					\big)
					(u,v)
					=\big( g \circ \Phi_{z,w})(u,v)
					\, .
                                      \end{equation*}
 This implies
$$
\widetilde D^{m,n}f(z,w)=D^{m,n}g(z,w)=D^{n,m}h(w,z)
=\gamma^{n-m}(D^{n,m}f)\circ\varrho_\gamma(w,z)
=\frac{z^{n-m}}{w^{n-m}}D^{n,m}f(z,w),
$$
where we have used the chain rule
  \begin{equation}
        \label{eq:PMinvarianceProof}
        D^{m,n}
        (f \circ \varrho_\gamma)
        =
        \gamma^{m-n}
        \cdot
        (D^{m,n} f)
        \circ
        \varrho_\gamma \,
    \end{equation}
in the last but one step.
\end{proof}

We give another proof of Proposition \ref{prop:DvsDtilde}, which is based on Fa\`a di Bruno's formula and which leads to the following  explicit expression for $D^{n,m} f$ in terms of the ``euclidean'' derivatives $\partial^j_z \partial^k_w f$.

	\begin{proposition}
				\label{rem:PMexplicit}

	Let $m,n \in \N_0$ with $m>0$ or $n>0$,  $U$ an open subset of $\Omega$ and $f \in C^{\infty}(U)$. Then
\begin{equation}
					\label{eq:AharonovGeneral}
					D^{n,m}f(z,w)
					=
					\sum \limits_{j=1}^n
					\sum \limits_{k=1}^{m}
					\partial_1^{j} \partial_2^k f (z,w)
					\big(1-zw\big)^{j+k}
					\frac{n!m!}{j!k!}
					\binom{n-1}{j-1}
					\binom{m-1}{k-1}
					(-w)^{n-j} (-z)^{m-k}
					\, .
                                      \end{equation}
                                      for all $(z,w) \in U \cap (\C^* \times \C^*)$.
                                      \end{proposition}

			\begin{proof}[Second proof of Proposition \ref{prop:DvsDtilde} and proof of Proposition \ref{rem:PMexplicit}]
                Fixing $z,w \in \C^*$, and setting
				\begin{equation*}
					\phi_{z,w}
					\in
					\mathcal{H}
					\big(
					\C \setminus \{-1/w\}
					\big)
					\, , \qquad
					\phi_{z,w}(u)
					\coloneqq
					\frac{z+u}{1+w u} \, ,
				\end{equation*}
				we may write \eqref{eq:DvsDtildeProof} as
				\begin{equation*}
					\big(
					\mathcal{F} \circ \Phi_{\mathcal{F}(z,w)}
					\big)
					(u,v)
					=
					\Big(
					\frac{z}{w} \cdot \phi_{w,z}(v),
					\frac{w}{z}\cdot \phi_{z,w}(u)
					\Big)
					\, .
				\end{equation*}
				If we denote by $B_{n,k}$ the Bell polynomials as discussed e.g.~in \cite[Sec.~4]{KS07diff}, then applying the formula of Fa\`a di Bruno as it can be found in \cite[p.~36]{Riordan} or \cite[p.~137]{Comtet} twice yields
				\begin{align*}
					&\widetilde{D}^{m,n} f(z,w)
					=
					\frac{\partial^m}{\partial u^m}
					\frac{\partial^n}{\partial v^n}
					\left[
					f
					\left(
					\frac{z}{w} \cdot \phi_{w,z}(v),
					\frac{w}{z} \cdot \phi_{z,w} (u)
					\right)
					\right]
					\at[\bigg]{(u,v)=(0,0)} \\
					&=
					\frac{\partial^m}{\partial u^m}
					\left[
					\sum \limits_{j=1}^n
					\frac{z^j}{w^j}
					\frac{\partial^j f}{\partial z^j}
					\Big(
					\frac{z}{w} \cdot \phi_{w,z}(v),
					\frac{w}{z} \cdot \phi_{z,w}(u)
					\Big)
					B_{n,j}
					\big(
					\phi_{w,z}'(v), \ldots, \phi_{w,z}^{(n-j+1)}(v)
					\big)
					\right]
					\at[\Bigg]{(u,v)=(0,0)} \\
					&=
					\sum \limits_{k=1}^m
					\sum \limits_{j=1}^n
					\Big( \frac{z}{w} \Big)^j
					\Big( \frac{w}{z} \Big)^k
					\partial^{j}_1 \partial^k_2 f (z,w)
					B_{m,k} \big( \phi_{z,w}'(0), \ldots, \phi_{z,w}^{(m-k+1)}(0) \big)
					B_{n,j} \big( \phi_{w,z}'(0), \ldots, \phi_{w,z}^{(n-j+1)}(0) \big)
				\end{align*}
				as well as
				\begin{align*}
					&D^{n,m} f(z,w)
					=
					\frac{\partial^n}{\partial u^n}
					\frac{\partial^m}{\partial v^m}
					\Big[
					f
					\big(
					\phi_{z,w}(u),
					\phi_{w,z}(v)
					\big)
					\Big]
					\at[\Big]{(u,v)=(0,0)} \\
					&=
					\frac{\partial^n}{\partial u^n}
					\left[
					\sum \limits_{k=1}^m
					\frac{\partial^k f}{\partial w^k}
					\big(
					\phi_{z,w}(u),
					\phi_{w,z}(v)
					\big)
					B_{m,k}
					\big( \phi_{w,z}'(v), \ldots, \phi_{w,z}^{(m-k+1)}(v) \big)
					\right]
					\at[\Bigg]{(u,v)=(0,0)} \\
					&=
					\sum \limits_{j=1}^n
					\sum \limits_{k=1}^m
					\partial^j_1 \partial^k_2 f (z,w)
					B_{n,j} \big( \phi_{z,w}'(0), \ldots, \phi_{z,w}^{(n-j+1)}(0) \big)
					B_{m,k} \big( \phi_{w,z}'(0), \ldots, \phi_{w,z}^{(m-k+1)}(0) \big) \,.
				\end{align*}
				By a simple induction argument, we have
				\begin{equation*}
					\phi^{(k)}_{z,w}(0)
					=
					(-1)^{k-1} k!
					w^{k-1}
					(1 - zw)
					\,, \qquad k \in \N \, .
				\end{equation*}
				Using moreover the well--known property (see e.g.~\cite[Lemma~4.3]{KS07diff} with $a = b = -1$)
				\begin{equation*}
					B_{n,j} \big(x_1,-x_2,\ldots, (-1)^{n-j} x_{n-j+1} \big)
					=
					(-1)^{n-j}
					B_{n,j}(x_1,x_2, \ldots, x_{n-k+1}) \, ,
				\end{equation*}
				we see that the conclusion of Proposition \ref{prop:DvsDtilde} would follow from
				\begin{align*}
					&\Big( \frac{z}{w} \Big)^j
					\Big( \frac{w}{z} \Big)^k
					B_{m,k} \big( 1, 2 w^1, \ldots, (m-k+1)! w^{m-k} \big)
					B_{n,j} \big(1 , 2 z^1,\ldots, (n-j+1)! z^{n-j} \big) \\
					& \qquad
					=
					\Big( \frac{z}{w} \Big)^n
					\Big( \frac{w}{z} \Big)^m
					B_{n,j} \big( 1, 2 w^1, \ldots, (n-j+1)! w^{n-j} \big)
					B_{m,k} \big( 1,2 z^1, \ldots, (m-k+1)! z^{m-k} \big)
					\, .
				\end{align*}
				This identity however holds since
				\begin{equation*}
					B_{m,k}
					\big(
					1, 2 w^1,\ldots, (m-k+1)! w^{m-k}
					\big)
					=
					w^{m-k}
					B_{m,k}\big(1,2!,\ldots, (m-k+1)!\big)
					\, ,
				\end{equation*}
				see \cite[Lemma 4.3]{KS07diff}. This completes the second proof of Proposition \ref{prop:DvsDtilde}.  The formula (\ref{eq:AharonovGeneral}) now follows from   the final formula in \cite[Sec.~4]{KS07diff},
				\begin{equation*}
					B_{m,k}
					\big(
					1,-2 , \ldots, (-1)^{m-k} (m-k+1)!
					\big)
					=
					(-1)^{m-k}
					\frac{m!}{k!}
					\binom{m-1}{k-1}
					\,.
				\end{equation*}
			\end{proof}

			We deduce the following simple consequence, which will play a crucial role in Section \ref{sec:StarProduct}, since it will allow us to define the Wick star product on $\mathcal{H}(\Omega)$.
			\begin{corollary}
				\label{cor:PeschlMindaChangeOfCoordinates}
				Let $m,n \in \N_0$, $U$ be an open subset of $\Omega$ and let $f \in C^{\infty}(U)$. Then
				\begin{equation}
					\label{eq:PMChangeOfCoordinates}
					D^{n,m}(f \circ \mathcal{F})
					\circ
					\mathcal{F}
					\at[\Big]{(z,w)}
					=
					D^{n,m}f_-
					\circ
					\mathcal{F}
					\at[\Big]{(z,w)}
					=
					\frac{z^m}{w^m}
					\frac{w^n}{z^n}
					\big( D^{m,n} f \big)(z,w)
				\end{equation}
				for all $(z,w) \in U \cap (\C^*\times \C^*)$.
			\end{corollary}
			\begin{proof}
				This is an immediate consequence of \eqref{eq:PMFlip}, Proposition~\ref{prop:DvsDtilde} and
				\begin{equation*}
					D^{n,m}f_-
					\circ
					\mathcal{F}
					\at[\Big]{(z,w)}
					=
					\partial_1^n \partial_2^m
					\big(f_+ \circ \mathcal{F} \circ \Phi_{\mathcal{F}(z,w)}\big)
					(0,0)
					=
					\widetilde{D}^{m,n}f(z,w)
					\, .
					\qedhere
				\end{equation*}
			\end{proof}

			\begin{remark}[Homogeneous cases]
				\label{rmk:PMZeroHomogeneous}%
				Corollary~\ref{cor:PeschlMindaChangeOfCoordinates} implies that for $n=m \in \N_0$ and any smooth $f \in C^{\infty}(\Omega)$, the Peschl--Minda derivative $D^{n,n} f$ can be globally defined and induces a  function $D^{n,n} f \in C^{\infty}(\Omega)$. For example, $D^{1,1}$ is a multiple of the Laplacian $\Delta_{z,w}$ of $\Omega$,
				\begin{equation}
					\label{eq:Laplacian}
					D^{1,1}
					=					(1-zw)^2 \partial_z\partial_w
					=					\tfrac{1}{4}
					\Delta_{z,w}\, .
				\end{equation}
				We are going to investigate the natural domains of definition for the ``pure'' operators $D^{n,0}$ and $D^{0,n}$ in Section~\ref{sec:PurePM}.
			\end{remark}

			One may interpret \eqref{eq:PMChangeOfCoordinates} as a transformation formula changing from the standard chart to the flip chart via $\mathcal{F} = \mathcal{F}^{-1}$. We next generalize this transformation formula to any element of the group $\mathcal{M}$.
			\begin{proposition}[$\mathcal{M}$--invariance]
				\label{prop:PMinvariance}%
				Let $U \subseteq \Omega \cap \C^2$ be open, $f\in C^{\infty}(U)$ and $\psi\in \Aut(\hat{\C})$ with $T$ given by \eqref{eq:MoebiusAction}. Then
				\begin{equation}
					\label{eq:PMinvariance}
					D^{m,n}
					\big(f \circ T\big)
					(z,w)
					=
					\bigg(
					\frac{\psi'(z)}{(1/\psi(1/w))'}
					\bigg)^{\frac{m-n}{2}}
					\big((D^{m,n} f) \circ T\big)(z,w)
				\end{equation}
				holds for $(z,w) \in U$ and $m,n \in \N_0$.
			\end{proposition}
			The meaning of the square root in \eqref{eq:PMinvariance} becomes apparent in the following proof, see \eqref{eq:squarerootexplicit}.
			\begin{proof}
				Let $\psi\in \Aut(\hat{\C})$ with $T$ given by \eqref{eq:MoebiusAction}. By Lemma~\ref{lem:MoebiusGenerators} there are $\alpha, \beta \in \C$ and $\gamma \in \C^*$ such that $T = \varrho_\gamma \circ \Phi_{\alpha, \beta}$. By \eqref{eq:PMinvarianceProof},
				\begin{align*}
					D^{m,n}
					\big(
					f \circ \Phi_{\alpha,\beta}
					\big)(z,w)
					&=
					\partial_1^m \partial_2^n
					\big(
					f
					\circ
					\Phi_{\alpha,\beta}
					\circ
					\Phi_{z,w}
					\big)(0,0) \\
					&\overset{\eqref{eq:MoebiusTComposition}}{=}
					\partial_1^m \partial_2^n
					\big(
					f
					\circ
					\varrho_{\frac{1 + \alpha w}{1 + \beta z}}
					\circ
					\Phi_{\Phi_{z,w}(\alpha,\beta)}
					\big)(0,0) \\
					&=
					D^{m,n}
					\Big(
					f
					\circ
					\varrho_{\frac{1 + \alpha w}{1 + \beta z}}
					\Big)
					\big(
					\Phi_{z,w}(\alpha,\beta)
					\big) \\
					&\overset{\eqref{eq:PMinvarianceProof}}{=}
					\bigg(
					\frac{1 + \alpha w}{1 + \beta z}
					\bigg)^{m-n}
					\Big(
					(D^{m,n} f)
					\circ
					\varrho_{\frac{1 + \alpha w}{1 + \beta z}}
					\Big)
					\big(\Phi_{z,w}(\alpha,\beta)\big) \\
					&=
					\bigg(
					\frac{1 + \alpha w}{1 + \beta z}
					\bigg)^{m-n}
					(D^{m,n} f)
					\Big(
					\big(
					\varrho_{\frac{1 + \alpha w}{1 + \beta z}}
					\circ
					\Phi_{\Phi_{z,w}(\alpha,\beta)}
					\big)(0,0)
					\Big) \\
					&\overset{\eqref{eq:MoebiusTComposition}}{=}
					\bigg(
					\frac{1 + \alpha w}{1 + \beta z}
					\bigg)^{m-n}
					(D^{m,n} f)
					\big(
					\Phi_{\alpha,\beta}(z,w)
					\big) \, .
				\end{align*}
				Consequently,
				\begin{equation*}
					D^{m,n}
					\big(
					f \circ T
					\big)(z,w)
					=
					\gamma^{m-n}
					\bigg(
					\frac{1 + \alpha w}{1 + \beta z}
					\bigg)^{m-n}
					\big( (D^{m,n} f) \circ T \big)
					(z,w) \, .
				\end{equation*}
				It remains to compute the derivatives of the functions
				\begin{equation*}
					t_{\gamma}(z)
					\coloneqq
					\gamma z
					\quad \text{and} \quad
					\varphi_{\alpha,\beta}(z)
					\coloneqq
					\frac{\alpha+z}{1+\beta z}
				\end{equation*}
				from \eqref{eq:MoebiusGeneratorsOneVariable} for fixed $\gamma \in \C^*$ and $(\alpha,\beta) \in \Omega \cap \C^2$. We note $1/\varphi_{\alpha,\beta}(1/z) = \varphi_{\beta,\alpha}(z)$,
				\begin{equation}
					\label{eq:MoebiusGeneratorsDerivatives}
					t_{\gamma}'(z)
					=
					\gamma
					\quad \text{and} \quad
					\frac{\varphi_{\alpha,\beta}'(z)}{(1/\varphi_{\alpha,\beta}(1/w))'}
					=
					\bigg(
					\frac{1+\alpha w}{1+\beta z}
					\bigg)^2 \, ,
				\end{equation}
				and hence
				\begin{equation}\label{eq:squarerootexplicit}
					\frac{\psi'(z)}{(1/\psi(1/w))'}
					=
					\frac{(t_{\gamma} \circ \varphi_{\alpha,\beta})'(z)}
					{(1/(t_{\gamma} \circ \varphi_{\alpha,\beta})(1/w))'}
					=
					\gamma^2
					\frac{\varphi_{\alpha,\beta}'(z)}{(1/\varphi_{\alpha,\beta}(1/w))'}
					=
					\gamma^2
					\bigg(
					\frac{1+\alpha w}{1+\beta z}
					\bigg)^2
					\, .
					\qedhere
				\end{equation}
			\end{proof}

			As a special case, we obtain homogeneity of the Peschl--Minda differential operators.
			\begin{corollary}[Homogeneity]
				\label{cor:PMHomogeneity}
				Let $m,n \in \N_0$, $k \in \Z$ and $f \in C^\infty(\Omega)$ be $k$--homogeneous, i.e.
				\begin{equation*}
					f \circ \varrho_\gamma
					=
					\gamma^k \cdot f
					\quad
					\text{for all}
					\quad
					\gamma \in \C^*
					\,
				\end{equation*}
				and the dilations $\varrho_\gamma$ from \eqref{eq:MoebiusGenerators}. Then the Peschl--Minda derivative $D^{m,n} f$ is $(k+m-n)$--homogeneous.
			\end{corollary}

			Note that choosing $\psi(z):=1/z$ yields the factor $\left(\psi'(z)/(1/\psi(1/w))'\right)^{(m-n)/2}=(w/z)^{m-n}$ in \eqref{eq:PMinvariance}. We have already encountered this prefactor when comparing $D^{m,n}f$ and $D^{n,m}(f\circ\mathcal{F})\circ\mathcal{F}$ in \eqref{eq:PMChangeOfCoordinates}. In view of the flip map $\mathcal{F}$, i.e. the precomposition of the flip with $D^{m,n}f$, only the inversion $(z,w)\mapsto(1/z,1/w)$ in $\mathcal{M}^+$ generates a prefactor, whereas interchanging $n$ and $m$ corresponds solely to swapping $z$ with $w$. Further, by virtue of Proposition~\ref{prop:DvsDtilde}, one gets similar formulas for $\widetilde{D}^{m,n}$. For instance, in the notation of Proposition~\ref{prop:PMinvariance},
			\begin{equation*}
				\widetilde{D}^{m,n}
				\big(f \circ T\big)
				(z,w)
				=
				\bigg(
				\psi(z)\psi(1/w)
				\frac{w}{z}
				\bigg)^{m-n}
				\bigg(
				\frac{\psi'(z)}{(1/\psi(1/w))'}
				\bigg)^{\frac{n-m}{2}}
				\big(
				(\widetilde{D}^{m,n} f)
				\circ T
				\big)(z,w) \, .
			\end{equation*}
			Combining Corollary~\ref{cor:PeschlMindaChangeOfCoordinates} with Proposition~\ref{prop:PMinvariance} moreover yields
			\begin{equation*}
				D^{m,n}
				\big(f \circ T \circ \mathcal{F} \big)
				(z,w)
				=
				\bigg(
				\frac{\psi'(w)}{(1/\psi(1/z))'}
				\bigg)^{\frac{n-m}{2}}
				\big((D^{n,m} f) \circ T \circ \mathcal{F}\big)(z,w)
				\, .
			\end{equation*}

			The main result of this section is that the operators \eqref{eq:PMStandard} can be generated from the Laplacian $\Delta_{z,w}=4 D^{1,1}$ of $\Omega$, see again \eqref{eq:Laplacian}, and the operators $D^{n,0}$ and $D^{0,n}$ for $n \in \N_0$, which we shall call \emph{pure Peschl--Minda differential operators}. Similarly, $D^{m,n}$ for $m,n > 0$ shall be referred to as $\emph{mixed}$ Peschl--Minda differential operators. We first note the following recursive identities, which might be interesting and useful in their own right.
			\begin{proposition}[Recursion identities]
				\label{prop:PMRecursion}
				Let $U \subseteq \Omega \cap \C^2$ be open and $m,n \in \N_0$. Then
				\begin{align}
					\label{eq:PMRecursionZ}
					D^{m+1,n}
					&=
					(n-m)w
					D^{m,n}
					+
					D^{1,0}
					\circ
					D^{m,n}
					+
					n(n-1)
					D^{m,n-1} \\
					\label{eq:PMRecursionW}
					D^{m,n+1}
					&=
					(m-n)z
					D^{m,n}
					+
					D^{0,1}
					\circ
					D^{m,n}
					+
					m(m-1)
					D^{m-1,n}
				\end{align}
				as operators on $C^\infty(U)$. Here we set $D^{j,k}=0$ if $j$ or $k$ is a negative integer.
			\end{proposition}
			\begin{proof}
				The trick is to compute
				\begin{equation*}
					\partial_1
					\big[
					D^{m,n}
					(f \circ \Phi_{z,w})
					\big]
					(0,0)
				\end{equation*}
				in two ways. On the one hand, using the $\mathcal{M}$--invariance \eqref{eq:PMinvariance} and \eqref{eq:MoebiusGeneratorsDerivatives} in the form
				\begin{equation*}
					D^{m,n}
					\big(
					f \circ \Phi_{z,w}
					\big)(u,0)
					=
					(1 + wu)^{n-m}
					\cdot
					(D^{m,n}f)
					\big(
					\Phi_{z,w}(u,0)
					\big) \, ,
				\end{equation*}
				we get
				\begin{align*}
					\partial_1
					\big[
					D^{m,n}
					(f \circ \Phi_{z,w})
					\big]
					(0,0)
					&=
					(n-m)w
					D^{m,n} f(z,w)
					+
					\partial_1
					\big[
					(D^{m,n}f)
					\circ
					\Phi_{z,w}
					\big]
					(0,0) \\
					&=
					(n-m)w
					D^{m,n} f(z,w)
					+
					\big(
					D^{1,0} \circ D^{m,n}
					\big)f
					(z,w) \, .
				\end{align*}
				On the other hand, the definition of $D^{m,n}$ leads to
				\begin{align*}
					&\partial_1
					\big[
					D^{m,n}
					(f \circ \Phi_{z,w})
					\big]
					(0,0)
					=
					\frac{d}{d u}
					\at[\bigg]{u=0}
					\frac{\partial^{m+n}}{\partial \hat{u}^m \partial \hat{v}^n}
					\bigg[
					\Big(
					f \circ \Phi_{z,w}
					\Big)
					\big(
					\Phi_{u,0}(\hat{u},\hat{v})
					\big)
					\bigg]
					\at[\bigg]{(\hat{u},\hat{v})=(0,0)} \\
					&=
					\frac{\partial^{m+n}}{\partial \hat{u}^m \partial \hat{v}^n}
					\frac{d}{d u}
					\at[\bigg]{u=0}
					\bigg[
					\Big(
					f \circ \Phi_{z,w}
					\Big)
					\bigg(
					u + \hat{u},
					\frac{\hat{v}}{1 + u \hat{v}}
					\bigg)
					\bigg]
					\at[\bigg]{(\hat{u},\hat{v})=(0,0)}  \\
					&=
					\frac{\partial^{m+n}}{\partial \hat{u}^m \partial \hat{v}^n}
					\Big[
					\partial_1
					\big(
					f \circ \Phi_{z,w}
					\big)(\hat{u},\hat{v})
					-
					\hat{v}^2
					\partial_2
					\big(
					f \circ \Phi_{z,w}
					\big)(\hat{u},\hat{v})
					\Big]
					\at[\bigg]{(\hat{u},\hat{v})=(0,0)}  \\
					&=
					\bigg(
					\frac{\partial^{m+n+1}}{\partial \hat{u}^{m+1} \partial \hat{v}^n}\at[\bigg]{(\hat{u},\hat{v})=(0,0)}
					\big(
					f \circ \Phi_{z,w}
					\big)(\hat{u},\hat{v})
					-
					n(n-1)
					\frac{\partial^{m+n-1}}{\partial \hat{u}^{m} \partial \hat{v}^{n-1}}
					\at[\bigg]{(\hat{u},\hat{v})=(0,0)}
					\big(
					f \circ \Phi_{z,w}
					\big)(\hat{u},\hat{v})
					\bigg)
					\\
					&=
					D^{m+1,n} f(z,w)
					-
					n(n-1)
					D^{m,n-1} f(z,w) \, ,
				\end{align*}
				where we have used the Leibniz rule for derivatives in the second last step. Note that $u,\hat{u}$ and $\hat{v}$ are just dummy variables and have no deeper meaning.
			\end{proof}

			Using Proposition~\ref{prop:PMRecursion}, one may give another proof of Corollary~\ref{cor:PMHomogeneity} by means of a simple induction. As before, Proposition~\ref{prop:DvsDtilde} yields similar recursion formulas for $\widetilde{D}^{m,n}$.

            \medskip

            In order to state the main result of this section, we construct a family of polynomials $P_{m,n}(x)$ for $m,n\in\N_0$.
            For $n \in \N_0$ define the polynomials $P_n$ recursively by $P_0(x) \coloneqq 1$, $P_1(x) \coloneqq x$ and
				\begin{equation}\label{eq:recursivepolynomial}
					P_{n+1}(x)
					\coloneqq
					\big(
					x+2n^2
					\big)
					P_{n}(x)
					-
					n^2
					(n-1)^2
					P_{n-1}(x)
					\,,
					\qquad n \ge 1 \, .
				\end{equation}
Let $\alpha_{0,0}=1,$ $\alpha_{n,0}=n!(n-1)!$ for $n\ge 1$ and
$$
\alpha_{n,k}=\frac{n!(n-1)!}{k!(k-1)!}\,, \qquad 1\le k\le n.
$$

Observe that the recurrence relation $\alpha_{n,k}=n(n-1)\alpha_{n-1,k}$
holds for $1\le k<n.$
We now define $P_{n+p,n}$ for $n\in\N$ recursively in $p\in\N_0.$
First, we set $P_{n,n}=P_n$ for $n\in\N_0.$
Having established $P_{n+p,n}$ for some $n\in\N_0$,
we set
$$
P_{n+p+1,n}(x)=\sum_{k=0}^n\alpha_{n,k}P_{k+p,k}(x)\,, \qquad
n\in\N_0\,.
$$
In this way, $P_{m,n}$ has been determined for $0\le n\le m.$
Finally, we set $P_{m,n}=P_{n,m}$ for $0\le m<n.$
We observe that $P_{m,n}(x)$ is a monic polynomial of degree $\min\{m,n\}$ with integer coefficients.
Note that the differential operator $P_{m,n}(D^{1,1})$ can be defined in a natural way.

We are now in~a position to state and prove the main theorem of this section.
			\begin{theorem}[Mixed Peschl--Minda differential operators]
				\label{thm:PMwithLaplace}
				Let $U \subseteq \Omega \cap \C^2$ be open and $m,n \in\N_0$.
				Then
				\begin{equation}
					\label{eq:PMwithLaplace}
					D^{m,n}
					=
					\begin{cases}
						D^{m-n,0}
						\circ
						P_{m,n}
						\big(D^{1,1}\big)
						&\text{ if }
						m \ge n \\[2mm]
						D^{0,n-m}
						\circ
						P_{m,n}\big( D^{1,1} \big)
						& \text{ if }
						n \ge m
					\end{cases}
				\end{equation}
				as operators on $C^\infty(U)$.
			\end{theorem}
			\begin{proof}
				Our strategy consists of a two--step induction. To this end, we prove \eqref{eq:PMwithLaplace} by showing that the claim
				\begin{equation*}
					\label{eq:PMwithLaplaceProof1}
					\eqref{eq:PMwithLaplace} \text{ for $D^{p+n,n}$ holds for all }
					n \in \N_0 \, .\tag{$S_p$}
				\end{equation*}
				holds for all $p \in \N_0$.

				\smallskip
				\textit{Base case: $(S_0)$, i.e. $p=0$:}
				First note that \eqref{eq:PMwithLaplace} is obvious for $n=0$ since $D^{0,0}$ is the identity operator. Thus, assume \eqref{eq:PMwithLaplace} for $D^{j,j}$ holds for all non-negative integers $j \le n \in\N,$ i.e.
				${D^{j,j}=P_{j,j}(D^{1,1})=P_{j}(D^{1,1})}$.
				Note that
				\begin{equation*}
					D^{1,0} \circ D^{0,1}
					=
					D^{1,1}
					-
					w D^{0,1}
					\, ,
				\end{equation*}
				and hence
				\begin{align}
					\nonumber
					D^{n+1,n+1}
					&\overset{\eqref{eq:PMRecursionZ}}{=}
					(w + D^{1,0}) D^{n,n+1}
					+
					n(n+1)
					D^{n-1,n-1} \\
					\nonumber
					&\overset{\eqref{eq:PMRecursionW}}{=}
					(w + D^{1,0})
					\big(
					D^{0,1}
					\circ
					D^{n,n}
					+
					n(n-1)
					D^{n-1,n}
					\big)
					+
					n(n+1)
					D^{n,n} \\
					\nonumber
					&=
					D^{1,1}
					\circ
					D^{n,n}
					+
					n(n-1)
					(w + D^{1,0})
					\circ
					D^{n-1,n}
					+
					n(n+1)
					D^{n,n} \\
					\nonumber
					&\overset{\eqref{eq:PMRecursionZ}}{=}
					D^{1,1}
					\circ
					D^{n,n}
					+
					n(n-1)
					\big(
					D^{n,n}
					-
					n(n-1)
					D^{n-1,n-1}
					\big)
					+
					n(n+1)
					D^{n,n} \\
					\label{eq:PMwithLaplaceHomogeneous}
					&=
					D^{1,1}
					\circ
					D^{n,n}
					+
					2n^2
					D^{n,n}
					-
					n^2(n-1)^2
					D^{n-1,n-1}
					\, .
				\end{align}
				By the induction hypothesis, we can write
					$$
					D^{n+1,n+1}
					=
					(D^{1,1}+2n^2)P_n(D^{1,1})-n^2(n-1)^2P_{n-1}(D^{1,1})
					=P_{n+1}(D^{1,1}).
					$$
 We have proved ($S_0$).

				\medskip

				\textit{Induction step: $(S_p) \implies (S_{p+1})$.}
				By the induction hypothesis, we may assume that we have proved \eqref{eq:PMwithLaplaceProof1} for some $p \in \N_0$. We claim that
				\begin{equation}
					\label{eq:PMwithLaplaceProof2}
					D^{n+p+1,n}
					=
					\sum_{k=0}^n
					\alpha_{n,k}
					\big(
					D^{1,0} - pw
					\big)
					D^{k+p,k} \,, \qquad n\in\N_0 \, .
				\end{equation}
				We show this by induction on $n.$
				For $n=0$ this is just \eqref{eq:PMRecursionZ} with $\alpha_{0,0} = 1$.
				We next assume that \eqref{eq:PMwithLaplaceProof2} is true for $n-1.$
				By \eqref{eq:PMRecursionZ} and the induction hypothesis, we get
				\begin{align*}
					D^{n+p+1,n}
					&=
					\big(
					-pw + D^{1,0}
					\big)
					D^{n+p,n}
					+
					n(n-1)
					D^{n+p,n-1} \\
					&=
					\big(
					-pw + D^{1,0}
					\big)
					D^{n+p,n}
					+
					n(n-1)
					\sum_{k=0}^{n-1}
					\alpha_{n-1,k}
					\big(
					D^{1,0} - pw
					\big)
					D^{k+p,k} \\
					&=
					\sum_{k=0}^n
					\alpha_{n,k}
					\big(
					D^{1,0} - pw
					\big)
					D^{k+p,k} \, ,
				\end{align*}
				as required. Hence, \eqref{eq:PMwithLaplaceProof2} is established. By virtue of \eqref{eq:PMwithLaplaceProof1}, \eqref{eq:PMwithLaplaceProof2} implies
				\begin{align}
					\nonumber
					D^{n+p+1,n}
					&=
					\sum_{k=0}^n
					\alpha_{n,k}
					\big(
					D^{1,0} - pw
					\big)
					D^{p,0}
					P_{p+k,k}\big(D^{1,1}\big) \\
					\nonumber
					&\overset{\eqref{eq:PMRecursionZ}}{=}
					\sum_{k=0}^n
					\alpha_{n,k}
					D^{p+1,0}
					P_{p+k,k}\big(D^{1,1}\big) \\
					&=
					\label{eq:PMwithLaplaceGeneral}
					D^{p+1,0}
					P_{k+p+1,k}\big(D^{1,1}\big) \, ,
				\end{align}
				which is $(S_{p+1})$. Hence we have proved that \eqref{eq:PMwithLaplaceProof1} holds for every $p \in \N_0$, and the first identity in (\ref{eq:PMwithLaplace}) is established.
			\end{proof}

			We may write for $m\ge n\ge1$
			$$
			P_{m,n}(x)=\sum_{k=0}^m a_k(m,n)x^k
			$$
			with $a_m(m,n)=1$ and $a_k(m,n)\in\Z.$
			It is easy to check that $a_0(m,n)=0.$
			Taking a closer look at the recursive formulas in the definition, one might wonder whether the coefficients $a_k(m,n)$ are positive integers. This is indeed the case, as we shall prove in the remainder of this section. By slight abuse of language, we say a polynomial has positive coefficients if all of its non--vanishing coefficients are positive.
			\begin{corollary}
				\label{cor:PMwithLaplaceHomogeneousPositivity}%
				For each $n \in \N,$ the polynomial $P_n(x)$ defined in \eqref{eq:recursivepolynomial} has positive integer coefficients.
			\end{corollary}
			\begin{proof}
				The assertion is trivial for $n=0.$ Thus we assume $n\ge 1.$
				Writing $$Q_n(x)=P_n(x)-n (n-1) P_{n-1}(x)$$ for $n \ge 1,$ the recursive definition of $P_n$ takes the form
				\begin{equation*}
					Q_{n+1}(x)
					=
					x P_n(x)
					+
					n(n-1)
					Q_{n}(x)
					\, ,
					\qquad n \in \N_0 \, .
				\end{equation*}
				Hence, if the polynomials $P_n$ and $Q_n$ have positive coefficients, then so does $Q_{n+1}$. But then the same is true for $P_{n+1}(x) = Q_{n+1}(x) + n(n-1) P_n(x)$. As the coefficient of $P_1(x) = x = Q_1(x)$ is certainly positive, this proves all coefficients of the polynomials $Q_n$ and $P_n$ are positive.
			\end{proof}
			\begin{corollary}\label{cor:PMwithLaplaceHomogeneousPositivityCoefficients}
				The polynomial $P_{m,n}(x)$ has positive coefficients for every pair
				of $m$ and $n$ in $\N_0.$
			\end{corollary}
			\begin{proof}
				By Corollary~\ref{cor:PMwithLaplaceHomogeneousPositivity}, the recursion \eqref{eq:PMwithLaplaceHomogeneous} preserves positivity. Since $\alpha_{n,k}>0$, \eqref{eq:PMwithLaplaceGeneral} also preserves positivity. Thus, our claim follows by induction.
			\end{proof}

			\begin{remark}
				Let $P_{n+1}(x)=\sum_{k=1}^{n+1}b_k(n+1) x^k$ be the polynomial from Corollary \ref{cor:PMwithLaplaceHomogeneousPositivity} where $n \in \N_0$ and $b_k(n) \in\C$, $k=1,...,n+1$. Using \eqref{eq:recursivepolynomial} it is possible to show that
				\begin{equation*}
					b_k(n+1)=\frac{(n-k+1)(2k^2+1)+3k}{3k}\left(\frac{n!}{(k-1)!}\right)^2+\sum_{j=k+1}^nb_{k-1}(j)\left(\frac{n!}{j!}\right)^2(n-j+1)\, .
				\end{equation*}
				This formula provides another proof of Corollary \ref{cor:PMwithLaplaceHomogeneousPositivityCoefficients}.
			\end{remark}

			\section{Pure Peschl--Minda Derivatives}
			\label{sec:PurePM}%
			In view of Theorem \ref{thm:PMwithLaplace}, the Peschl--Minda differential operators $D^{n,0}$ and $D^{0,n}$ play a distinguished role. Henceforth, we shall therefore write
			\begin{equation*}
				D_z^n
				\coloneqq
				D^{n,0}
				\quad \text{as well as} \quad
				D_w^n
				\coloneqq
				D^{0,n}
				\, ,
				\qquad
				n \in \N_0
			\end{equation*}
			and call $D_z^n$ and $D_w^n$ the \textit{pure Peschl--Minda differential operators of order $n$}. In this section we investigate their connection to the euclidean derivatives $\partial_z$ and $\partial_w$, we determine the natural domains of these operators and, in particular, we will explicitly describe their kernels. The crucial tool is the set of recursion identities from Proposition~\ref{prop:PMRecursion}. For the pure Peschl--Minda differential operators these identities reduce to
			\begin{equation}
				\label{eq:PMRecursionPure}
				D_z^{n+1}
				=
				D^1_z \circ D^n_z - nw D^{n}_z
				\quad \text{and} \quad
				D_w^{n+1}
				=
				D^1_w \circ D^n_w - nz D^{n}_w
				\, .
			\end{equation}
			This simpler recursion has the following useful consequences.
			\begin{lemma}[Pure Peschl--Minda differential operators of the same type commute]
				For all $m,n \in \N_0$,
				\begin{equation}
					\label{eq:PMcommute}
					D^n_z \circ D^m_z
					=
					D^m_z \circ D^n_z
					\quad \text{and} \quad
					D^n_w \circ D^m_w
					=
					D^m_w \circ D^n_w
					\, .
				\end{equation}
			\end{lemma}

			\begin{theorem}[Pure Peschl--Minda derivatives and Wirtinger derivatives]
				\label{thm:PMPureEuclidean}%
				Let $U \subseteq \Omega \cap \C^2$ be an open set, and let $f \in C^{\infty}(U)$. Then
				\begin{equation}
					\label{eq:PMPureExplicit}
					\begin{array}{rcl}
						D_z^{n+1} f(z,w)
						&=&
						\left( 1-z w \right)
						\partial_z^{n+1}
						\left[
						\left(1-z w\right)^n
						f(z,w)
						\right]
						\, , \\[2mm]
						D_{w}^{n+1} f(z,w)
						&=&
						\left(1-zw\right)
						\partial_w^{n+1}
						\left[
						\left(1-z w \right)^n
						f(z,w)
						\right]
					\end{array}
				\end{equation}
				for all $n \in \N_0$.
			\end{theorem}
			\begin{proof}
				The case $n=0$ follows from \eqref{eq:PMFirstOrder}. Assuming that (\ref{eq:PMPureExplicit}) holds for a fixed integer $n$, and combining \eqref{eq:PMRecursionPure} with \eqref{eq:PMcommute} and \eqref{eq:PMPureExplicit} for $n$, we obtain
				\begin{align*}
					&D_z^{n+1} f(z,w)
					=
					\big(
					D_z^n \circ D_z^1
					\big)
					f(z,w)
					-
					nw D_z^n f(z,w) \\
					&=
					\left( 1-z w \right)
					\partial_z^{n}
					\left[
					\left(1-z w\right)^{n}
					\partial_z f(z,w)
					\right]
					-
					nw D_z^n f(z,w) \\
					&=
					\left( 1-z w \right)
					\partial_z^{n+1}
					\left[
					\left(1-z w\right)^{n}
					f(z,w)
					\right]
					-
					\left( 1-z w \right)
					\partial_z^n
					\left[
					(-w)n
					\left(1-z w\right)^{n-1}
					f(z,w)
					\right]
					-
					nw D_z^n f(z,w) \\
					&=
					\left( 1-z w \right)
					\partial_z^{n+1}
					\left[
					\left(1-z w\right)^{n}
					f(z,w)
					\right] \, .
					\qedhere
				\end{align*}
			\end{proof}

			We next pose and answer the following foundational  question:
			where do the Peschl--Minda operators want to live?
			Recall from Definition~\ref{def:PeschlMindageneral} that we have initially defined
			the Peschl--Minda derivatives $D^{m,n} f(z,w)$ in local coordinates $(z,w) \in \Omega \cap \C^2$. In particular, for any open set $U \subseteq \Omega \cap \C^2$, the Peschl--Minda operator $D^{m,n}$ acts as a linear differential operator on the algebra $C^{\infty}(U)$. However, the transformation rule \eqref{eq:PMChangeOfCoordinates} implies that in general $D^{m,n}$ does \textit{not} act as an operator on $C^{\infty}(\Omega)$, except in the homogeneous case $n=m$, see again Remark~\ref{rmk:PMZeroHomogeneous}.
			Another glance at the transformation rule \eqref{eq:PMChangeOfCoordinates} in Corollary~\ref{cor:PeschlMindaChangeOfCoordinates} for the
			\textit{pure} Peschl--Minda derivative $D^n_z f$ shows that
			\begin{equation}
				\label{eq:PMChangeOfCoordinatesPure}
				D^n_z f(z,w)
				=
				\frac{w^n}{z^n}
				\,
				D^n_w f_-
				\big(
				1/w,1/z
				\big)
				\quad
				\text{ for all }
				(z,w) \in U \, .
			\end{equation}
			Hence, if we assume that $f \in C^{\infty}(U)$ for some open set $U \subseteq \Omega$ that contains the points $(\infty,w)$, $w \in \C^{*}$, then \eqref{eq:PMChangeOfCoordinatesPure} inevitably suggests to \textit{define} $D^n_z f(\infty,w) \coloneqq 0$ for $w \in \C^*$. This way we have extended the pure Peschl--Minda differential operator $D^n_z$ to an operator acting on $f \in C^{\infty}(U)$ for all open sets $U \subseteq \Omega_+$, where
			\begin{equation*}
				\Omega_+
				\coloneqq
				\Omega
				\setminus
				\big\{
				(z,\infty)
				\colon
				z \in \hat{\C}
				\big\} \, .
			\end{equation*}
			Similarly, we extend the definition of $D^n_w f$ to all open subsets of the domain
			\begin{equation*}
				\Omega_-
				\coloneqq
				\Omega
				\setminus
				\big\{
				(\infty, w)
				\colon
				w \in \hat{\C}
				\big\}
			\end{equation*}
			by setting $D^n_w f(z,\infty) \coloneqq 0$ for $z \in \C^*$. The subdomains $\Omega_+$ and $\Omega_-$ of $\Omega$ are visualized in Figure~\ref{fig:OmegaPM}. The edges of each square represent points near infinity. If the edge belongs to the domain, it is dashed. The blue dots correspond to boundary points. Note that points on opposite edges and in particular the four corners are identified. We note in passing that the domains~$\Omega_+$ and~$\Omega_-$ also play a central role for the spectral theory of the invariant Laplacian of the unit disk and the Riemann sphere in \cite{HeinsMouchaRoth2} as well as for the Fr\'echet space structure of $\mathcal{H}(\Omega)$, see \cite{HeinsMouchaRoth3}.

			\begin{figure}[h]
				\begin{minipage}{6\textwidth/19}
					\centering
					\includegraphics[width = \textwidth]{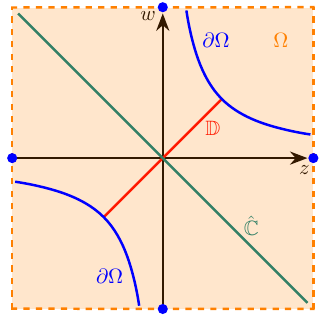}
				\end{minipage}
				\begin{minipage}{6\textwidth/19}
					\centering
					\includegraphics[width = \textwidth]{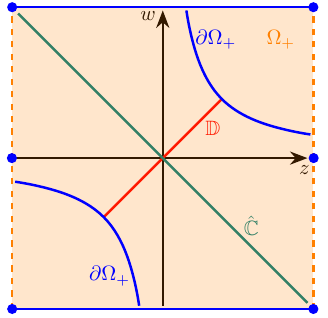}
				\end{minipage}
				\begin{minipage}{6\textwidth/19}
					\centering
					\includegraphics[width = \textwidth]{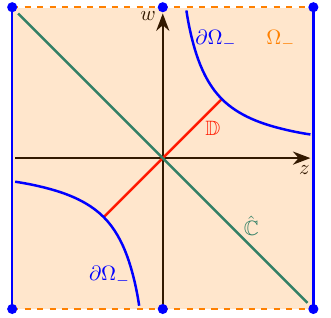}
				\end{minipage}
				\caption{Schematic picture of the domains $\Omega$ (left), $\Omega_+$ (center) and $\Omega_-$ (right) with points at infinity.}
				\label{fig:OmegaPM}
			\end{figure}

			\medskip

			Our next theorem is a key result for understanding the pure Peschl--Minda operators. Consider the maps
			\begin{equation*}
				\Psi_{\pm}
				\colon
				\Omega_{\pm} \longrightarrow \C^2
				\, , \qquad
				\Psi_+(z,w)
				\coloneqq
				\Big(
				\frac{z}{1-zw}, w
				\Big)
				\quad \text{ resp. } \quad
				\Psi_-(z,w)
				\coloneqq
				\Big(
				z,\frac{w}{1-zw}
				\Big) \, ,
			\end{equation*}
			which are easily seen to be biholomorphic with inverses
			\begin{equation*}
				\Psi_+^{-1}
				(u,v)
				=
				\Big(
				\frac{u}{1+uv}, v
				\Big)
				\quad \text{and} \quad
				\Psi_-^{-1}
				(u,v)
				=
				\Big(
				u, \frac{v}{1+uv}
				\Big)
				\, .
			\end{equation*}
			This implies that $\Omega_\pm$ are simply connected subdomains of $\Omega$. Furthermore, in $\Psi_{\pm}$--coordinates the corresponding $n$--th pure Peschl--Minda derivative turns out to be a multiple of the $n$--th euclidean derivative.
			\begin{theorem}[Global linearisation of pure Peschl--Minda derivatives]
				\label{thm:PMLinearisation}
				Let $n \in \N$. If $U$ is an open subset of $\Omega_+$ and $f \in C^{\infty}(U)$, then
				\begin{equation}
					\label{eq:PMLinearisationZ}
					D^n_z
					f(z,w)
					=
					(1-zw)^{-n}
					\partial^n_1
					\big( f \circ \Psi^{-1}_+ \big)
					( \Psi_+(z,w) )
					\quad \text{ for all } (z,w) \in U \,;
				\end{equation}
				If $U$ is an open subset of $\Omega_-$ and $f \in C^{\infty}(U)$, then
				\begin{equation*}
					D^n_w
					f(z,w)
					=
					(1-zw)^{-n}
					\partial^n_2
					\big( f \circ \Psi^{-1}_- \big)
					(\Psi_-(z,w))
					\quad \text{ for all } (z,w) \in U \, .
				\end{equation*}
				In particular, $\mathcal{H}(\Omega_+)$ is a $D^n_z$--invariant subspace of $C^{\infty}(\Omega_+)$ and $\mathcal{H}(\Omega_-)$ is a $D^n_z$--invariant subspace of $C^{\infty}(\Omega_-)$.
			\end{theorem}
			Considering $\Psi_+$ as a chart of the manifold $\Omega_+$ the expression $\partial^n_1 \big( f \circ \Psi^{-1}_+ \big)( \Psi_+(z,w) )$ is simply the $n$--th partial derivative of $f$ with respect to the first component.
			\begin{proof}
				For $(\infty,w) \in \Omega_+$, both sides of \eqref{eq:PMLinearisationZ} simply vanish: On the right hand side, this is due to the prefactor $(1-zw)^{-n}$ and the fact that the partial derivatives of $f$ with respect to $\Psi_+$ map into $\C$. On the left hand side we get $0$ by the very definition of $D^n_z f(\infty,w)$. It thus remains to consider $(z,w) \in U \cap \C^2$, for which we proceed by induction. The case $n=1$ follows from \eqref{eq:PMFirstOrder} via
				\begin{equation*}
					\partial_1
					\big( f \circ \Psi^{-1}_+ \big)
					\at[\Big]{\Psi_+(z,w)}
					=
					\frac{1}{(1+uv)^2}
					\at[\bigg]{\Psi_+(z,w)}
					\partial_z f(z,w)
					=
					(1-zw)^2
					\partial_z f(z,w)
					=
					(1-zw)
					D_z^1 f(z,w)
					\, ,
				\end{equation*}
				where we may use the chain rule by virtue of $(z,w) \in \C^2$. Assuming that \eqref{eq:PMLinearisationZ} holds for some $n \in \N$, we get
				\begin{align*}
					\frac{\partial}{\partial z}
					\left((1-zw)^n D^n_z f (z,w)\right)
					&=
					\frac{\partial}{\partial z}
					\left[
					\partial^n_1 \left(f \circ \Psi_+^{-1}\right)(\Psi_+(z,w))
					\right] \\
					&=
					(1-zw)^{-2}
					\partial^{n+1}_1
					\left( f \circ \Psi_+^{-1} \right)
					(\Psi_+(z,w))\, .
				\end{align*}
				In view of the recursion identity \eqref{eq:PMRecursionPure}, this implies
				\begin{align*}
					D^{n+1}_z f(z,w)
					&=
					-
					nw
					D^n_z f(z,w)
					+
					\big( D^1_z (D^n_z f) \big)
					(z,w) \\
					&=
					-
					nw
					D^n_z f(z,w)
					+
					(1-zw) \frac{\partial}{\partial z}
					\left( (1-zw)^{-n} (1-zw)^n D^n_z f(z,w) \right) \\
					&=
					-
					nw
					D^n_z f(z,w)
					+
					nw D_z^n f(z,w)
					+
					(1-zw)^{-n-1}
					\partial^{n+1}_1
					\left( f \circ \Psi_+^{-1} \right)
					(\Psi_+(z,w)) \\
					&=
					(1-zw)^{-n-1}
					\partial^{n+1}_1
					\left( f \circ \Psi_+^{-1} \right)
					(\Psi_+(z,w))
					\, ,
				\end{align*}
				which proves \eqref{eq:PMLinearisationZ}. The additional statements are clear, see also Lemma~\ref{lem:PMHolomorphic}.
			\end{proof}

			Theorem \ref{thm:PMPureEuclidean} and Theorem \ref{thm:PMLinearisation} together allow us to determine the kernels of the pure Peschl--Minda differential operators explicitly. We denote by $\spa_{\pm} M$ the closures of the linear span of the set $M$ in the Fr\'echet spaces $\mathcal{H}(\Omega_\pm)$.
			\begin{corollary}
				\label{cor:PMKernel}
				Let $n \in \N_0$. Then
				\begin{eqnarray}
					\label{eq:kerDn}
					\ker
					\left(
					D_z^{n+1}
					\at{\mathcal{H}(\Omega_+)}
					\right)
					&=&
					{\spa}_+\left\{ \frac{z^j w^k}{(1-zw)^n} \, ; \, 0 \le j \le n, 0 \le k<\infty
					\right\} \, , \\  \label{eq:kerDn2}
					\ker
					\left(
					D_{w}^{n+1}
					\at{\mathcal{H}(\Omega_-)}
					\right)
					&=& {\spa}_-\left\{ \frac{z^j w^k}{(1-zw)^n} \, ; \, 0 \le k \le n, 0 \le j<\infty
					\right\} \, .
				\end{eqnarray}
			\end{corollary}

			\begin{proof}
				First note that for $j \le n$ and $k \in \N_0$ the limits
				\begin{equation*}
					\lim_{z \rightarrow \infty}
					\frac{z^j w^k}{(1-zw)^n}
					=
					\lim_{z \rightarrow \infty}
					\frac{z^j}{z^n}
					\frac{w^k}{(1/z - w)^n}
					=
					\delta_{j,n}
					(-1)^n
					\frac{w^k}{w^n}
					\, ,
					\qquad
					w \in \C^*
				\end{equation*}
				exist and thus $z^jw^k/(1-z w)^n$ has a holomorphic extension to $\Omega_+$. Now, Theorem~\ref{thm:PMPureEuclidean} shows that \eqref{eq:kerDn} is equivalent to
				\begin{equation*}
					\left\{
					g \in \mathcal{H}(\C^2)
					\colon
					\partial_z^{n+1}
					g=0
					\right\}
					=
					\left\{
					\sum \limits_{j=0}^n z^j h_j(w)
					\colon
					h_0, \ldots, h_n
					\in \mathcal{H}(\C)
					\right\}
					\, ,
				\end{equation*}
				which is well--known and otherwise readily verified.
			\end{proof}

			The spanning functions occurring on the right hand sides of \eqref{eq:kerDn} and (\ref{eq:kerDn2}) play a similar role for $\mathcal{H}(\Omega_\pm)$ as the monomials play for $\mathcal{H}(\C^2)$, but they are in general not linearly independent. In contrast,
			the functions
			\begin{equation*}
				f_{p,q}(z,w)
				\coloneqq
				\frac{z^p w^q}{(1-zw)^{\max\{p,q\}}}
				\, , \qquad
				(z,w) \in \Omega
			\end{equation*}
			are linearly independent, and in addition they do form a Schauder basis of $\mathcal{H}(\Omega)$.
			For a detailed discussion of this we refer to \cite[Sec.~4]{HeinsMouchaRoth3}. The functions $f_{p,q}$ have also been instrumental  in the original approach  to strict deformation quantization of the Poincar\'e disk in \cite{KrausRothSchoetzWaldmann2019}. This aspect will be discussed in more detail in Section \ref{sec:StarProduct} below.

			\section{Comparison with invariant derivatives in one variable}
			\label{sec:PMClassical}
			In this section we briefly indicate how previous work on invariant derivatives for functions of one variable naturally fits into the more general framework of differential operators for functions of two complex variables which we have developed in Section \ref{sec:PM} and Section \ref{sec:PurePM}. Moreover, we demonstrate how our two--variable approach provides additional insights concerning the original Peschl--Minda operators for functions of one variable.

			\medskip

			Let us begin by recalling the standard definition of the ``classical'' Peschl--Minda invariant derivatives for holomorphic functions of one variable, see \cite{KS07diff,Minda,Peschl1955,Schippers1999}.
			Let $f \in \mathcal{H}(\D)$ and fix $z \in \D$. Then the function $u \mapsto f\big(\tfrac{z+u}{1+\cc{z}u})$ is holomorphic in $\D$ and thus has a convergent Taylor expansion
			\begin{equation*}
				f
				\Big(
				\frac{z+u}{1+\overline{z} u}
				\Big)
				=
				\sum \limits_{n=0}^{\infty}
				\frac{D^nf(z)}{n!} u^n
				\qquad
				\text{for all} \;
				u \in \D\,.
			\end{equation*}
			This defines the (classical) Peschl--Minda derivatives $D^n f$ for holomorphic functions {$f:\D\to\C$}. Note that, in general, $D^nf$ does not belong to $\mathcal{H}(\D)$ and hence one may not iterate this process. It follows at once from  Lemma~\ref{lem:PMHolomorphic} that $D^n f(z) =D^{n}_z f(z,\overline{z})$ for any $f \in \mathcal{H}(\D)$, so the theory developed in the present paper also immediately applies to $D^n$. For instance, Aharonov \cite[(3.2)]{Ahar69} has proved the explicit formula
			\begin{equation}
				\label{eq:aharonov}
				D^n f(z)
				=
				\sum \limits_{k=1}^n
				\frac{n!}{k!}
				\binom{n-1}{k-1}
				\left( -\overline{z} \right)^{n-k}
				\left(1-|z|^2 \right)^k
				f^{(k)}(z) \, ,
				\qquad
				z \in \D,
			\end{equation}
			which we recognize as a special case of \eqref{eq:AharonovGeneral}. Further, the operators $D^n$ obey the following recursive identity,
			\begin{equation}
				\label{eq:KM}
				\begin{array}{rcl}
					D^1 f(z)
					&=&
					\left( 1-|z|^2 \right)
					f'(z)
					\, , \\[2mm]
					D^{n+1}f(z)
					&=&
					\left( 1-|z|^2 \right)
					\partial (D^nf) (z)
                    -
                    n \overline{z} D^n f(z) \, ,
				\end{array}
			\end{equation}
			which is exactly \eqref{eq:PMRecursionPure}. The recursion formula \eqref{eq:KM}, which already appears in \cite{KS07diff, Minda,Nakahara2003,Schippers1999,Schippers2003b}, makes it possible to define $D^n f$ for smooth functions by simply replacing the complex derivative~$'$ by the Wirtinger derivative $\partial$. Then Aharonov's explicit formula \eqref{eq:aharonov} still holds with $'$ replaced by $\partial$. Equipped with this definition of $D^nf$ for smooth functions $f$, it is just a short step to define the conjugate operator $\overline{D}$. In order to guarantee that the “three–bar rule” $\cc{\partial f} = \cc{\partial} (\cc{f})$ holds for the differential operators $D^n$ and $\overline{D}^n$  as well, we simply set
			\begin{equation*}
				\cc{D}^n f
				\coloneqq
				\cc{D^n(\cc{f})}
				\, .
			\end{equation*}
			It is not difficult to check that $\overline{D}^n$ coincides with $D_w^n$ for $w = \overline{z}$, so the operator $\overline{D}^n$ and its properties are included in our previous discussion. Recall, for instance, Proposition~\ref{prop:PMinvariance}. In terms of $D^n$ and $\overline{D}^n$, it states that for every
			smooth function $f \colon \D \longrightarrow \C$ and every $T \in \Aut(\D)$,
			\begin{equation*}
				D^n (f \circ T)
				=
				\left(
				\frac{T'}{\abs{T'}}
				\right)^n
				D^n f \circ T
				\, , \qquad
				\cc{D}^n (f \circ T)
				=
				\left(
				\frac{\overline{T'}}{\abs{T'}}
				\right)^n
				\cc{D}^n f \circ T
				\, .
			\end{equation*}
			The first identity is the well--known invariance property of $D^n$, see e.g.~\cite[Formula (4.1)]{Schippers1999}, \cite[Lemma~3.2]{KS07diff}. It follows that the absolute values $\abs{D^nf}$ and $\abs{\overline{D}^n f}$ as well as the product $D^n f \cdot \overline{D}^n g$ are absolute conformal invariants. Theorem~\ref{thm:PMPureEuclidean} includes the following ``explicit'' formula for the invariant operators $D^n$ and $\cc{D}^n$, which appears to be new.
			\begin{corollary}[$D^n$ and $\cc{D}^n$ in terms of $\partial$ and $\overline{\partial}$]
				\label{cor:PMClassicalEuclidean}
				Let $U \subseteq \D$ be an open set and let $f \colon U \longrightarrow \C$ be a smooth function. Then
				\begin{equation*}
					\begin{array}{rcl}
						D^{n+1} f(z)
						&=&
						\left( 1-|z|^2 \right)
						\partial^{n+1}
						\left[ \left(1-|z|^2 \right)^n f(z) \right]
						\, , \\[2mm]
						\overline{D}^{n+1} f(z)
						&=&
						\left(1-|z|^2\right)
						\cc{\partial}^{n+1}
						\left[ \left(1-|z|^2 \right)^n f(z) \right]
						\, .
					\end{array}
				\end{equation*}
			\end{corollary}

			We finally take a quick look at the global linearisation result of Theorem~\ref{thm:PMLinearisation} in terms of $D^n$. In fact, Theorem \ref{thm:PMLinearisation} suggests to consider the operator
			\begin{equation*}
				\hat{D}f(z)
				\coloneqq
				\left( 1-|z|^2\right)^2
				\partial_z f(z)
				=
				(1-|z|^2) D^1 f(z) \,
			\end{equation*}
			and its iterates $\hat{D}^n f \coloneqq \underbrace{\hat{D} \circ \ldots \circ \hat{D}}_{n\text{--times}}f$.
			Then, as a special instance of \eqref{eq:PMLinearisationZ}, we see that
			\begin{equation}
				\label{eq:PMLinearisationOneVariable}
				(1-|z|^2)^n
				D^n f(z)
				=
				\hat{D}^n f(z) \, .
			\end{equation}
			This identity has been noticed before, see Bauer \cite[p.~71]{BauerI} (the operator $\hat{D}$ is called $\delta$ there) and also Kim \& Sugawa \cite[Lemma~5.1]{KS07diff} (our $\hat{D}$ is their $d_\rho$). However, the Peschl--Minda derivatives on $\Omega$ provide a new point of view on the identity \eqref{eq:PMLinearisationOneVariable}: In our two--variable approach $\hat{D}_z f(z,w)$ is simply the first euclidean derivative of $f$ in the $\Psi_+$--coordinates w.r.t. the first variable, which gives a clear and immediate explanation why $\hat{D}$ is a natural object to study and why \eqref{eq:PMLinearisationOneVariable} does hold.

			\section{The Wick Star Product on $\mathcal{H}(\Omega)$}
			\label{sec:StarProduct}%
			The purpose of this section is to define and study the holomorphic analogue of the canonical Wick star product of the Poincaré disk $\D$, see e.g. \cite{BordemannBrischleEmmrichWaldmann, CahenGuttRawnsley1994}, for the manifold $\Omega$. Our main result is an explicit formula for the Wick star product on $\Omega$ in terms of the pure Peschl--Minda differential operators in Theorem~\ref{thm:StarProduct}. The previously known convergent star products on $\D$ and~$\hat{\C}$ are simple corollaries of this convergent star product on $\Omega$. The proofs rely on the results of Sections \ref{sec:PM}--\ref{sec:PMClassical}, the theory of Peschl--Minda derivatives for smooth functions on Riemann surfaces developed in
            \cite{Nakahara2003, Minda, Schippers1999, Schippers2003b, KS07diff}, and the formula for the Wick star product derived by Schmitt and Schötz in \cite[(5.24)]{SchmittSchoetz2022}. We also show that the star product on $\Omega$ is invariant with respect to the full M\"obius--type group $\mathcal{M}$; the previously known $\Aut(\D)$-- resp.~$\Aut(\hat{\C})$--invariance properties of the star product on $\D$ resp.~$\hat{\C}$ follow immediately by "restriction" to the diagonal resp.~anti--diagonal in $\Omega$.
			As other corollaries we obtain asymptotic expansions of the Wick star products on $\Omega$, $\D$ and $\hat{\C}$ with respect to the deformation parameter $\hbar$ as $\hbar \to 0$ within every sector excluding the negative real axis. The special case of the asymptotic expansion formula for the star product on the disk provides a considerable extension of Theorem 4.5 in \cite{KrausRothSchoetzWaldmann2019}, which handles the case of first order approximation for $\hbar \to 0+$ on the real axis.

			\medskip

			The bridge between \cite{Minda,Schippers1999, KS07diff} and \cite{SchmittSchoetz2022} consists in the notion of symmetrized covariant derivatives, which are closely related to the Peschl--Minda differential operators. In fact, they already implicitly appear in the construction devised in \cite{Minda,Schippers1999, KS07diff}. We briefly recall a special case of symmetrized covariant derivatives we shall need, which we formulate in the language of holomorphic Riemannian geometry. For a proper treatment of the latter, we refer e.g. to the textbook \cite{Manin1988}.

			\medskip

			We begin by fixing some notation. The tangent and cotangent bundles of $\Omega$ are denoted by $T\Omega$ and $T^* \Omega$, respectively. The holomorphic symmetric algebra bundle over $T^*\Omega$ is ${S(T^*\Omega) \coloneqq \bigoplus_{k=0}^\infty S^k T^*\Omega}$ with $S^0 T^* \Omega \coloneqq \Omega$ and $S^k T^*\Omega \coloneqq (T^* \Omega)^{\vee k}$ with the symmetric tensor product $\vee$. We write $\Gamma(E)$ for the holomorphic sections of the various holomorphic vector bundles $E$ we have just introduced. Finally, pullbacks of differential forms along a map $f$ shall be denoted by $f^*$. For functions, $f^*$ simply acts by precomposition.

			\medskip

			We begin the construction by equipping $\Omega$ with a holomorphic metric $g_\Omega$ (i.e. a non--degenerate, symmetric and holomorphic two--form) by defining
			\begin{equation}
				\label{eq:OmegaKaehler}
				g_\Omega
				\at[\Big]{\Omega \cap \C^2}
				(z,w)
				\coloneqq
				\frac{dz \vee dw}{(1-zw)^2}
				\quad \text{and} \quad
				g_\Omega
				\at[\Big]{\Omega \cap (\hat{\C}^* \times \hat{\C}^*)}
				(z,w)
				=
				\frac{dz_F \vee dw_F}{(1-z_F w_F)^2}
				\, ,
			\end{equation}
			where we denote the components of the standard chart \eqref{eq:ChartStandard} by $(z,w)$ and the ones of the flip chart \eqref{eq:ChartFlip} by $(z_F, w_F)$. This yields a well--defined extension $g_\Omega \in \Gamma(S^2 T^* \Omega)$ of the usual hyperbolic metric $g_\D$ and the spherical metric $g_{\hat{\C}}$ in the sense that $d_\D^* g_\Omega = g_\D$ and $d_{\hat{\C}}^* g_\Omega = - g_{\hat{\C}}$ with the diagonal maps
			\begin{equation}
				\label{eq:Diagonal}
				d_\D
				\colon
				\D
				\longrightarrow
				\D^2, \quad
				d_\D(z)
				=
				(z,\cc{z})
				\quad \text{and} \quad
				d_{\hat{\C}}
				\colon
				\hat{\C}
				\longrightarrow
				\hat{\C}^2, \quad
				d_{\hat{\C}}(z)
				=
				(z,-\cc{z})
			\end{equation}
			which map $\D$ resp.~$\hat{\C}$ into $\Omega$.
			This endows $\Omega$ with the structure of a holomorpic Kähler manifold. Incidentally, the group $\mathcal{M}$ of M\"obius--type transformations in $\Aut(\Omega)$ turns out to be precisely the subgroup of those holomorphic automorphisms of $\Omega$ which  preserve the metric $g_\Omega$, see \cite[Thm.~6.1]{HeinsMouchaRoth3}.

			\medskip

			Associated to  a metric $g_\Omega$, one may construct a Levi--Civita covariant derivative $\nabla$ by the usual arguments mutatis mutandis. This yields a canonical metric and torsion--free covariant derivative on $\Omega$ and its open subsets. We write shorthand $\nabla_z \coloneqq \nabla_{\partial_z}$ and $\nabla_w \coloneqq \nabla_{\partial_w}$. Its crucial feature is that the only non--vanishing Christoffel symbols in the standard chart are $\Gamma_z \coloneqq dz(\nabla_{\! z} \del_z)$ and $\Gamma_w \coloneqq dw(\nabla_{\! w} \del_w)$. This is a straightforward consequence of Kähler metrics being of type $(1,1)$.

			\medskip

			Moving forward, we associate a symmetrized covariant derivative $\SymD$ to our $\nabla$. To this end, we extend $\nabla$ to the cotangent bundle $T^*\Omega$ by demanding the compatibility with pairings
			\begin{equation}
				\label{eq:CovPairing}
				\nabla_{\! Y}
				\big(
				\alpha(X)
				\big)
				=
				\big(
				\nabla_{\! Y} \alpha
				\big)(X)
				+
				\alpha
				\big(
				\nabla_{\! Y} X
				\big)
			\end{equation}
			for $X, Y \in \Gamma(T\Omega)$ and $\alpha \in \Gamma(T^*\Omega)$. Imposing a Leibniz rule for the symmetric tensor product extends it further to the sections of the symmetric algebra bundle over $T^*\Omega$. The corresponding (holomorphic) \textit{symmetrized covariant derivative} $\SymD$ is then defined as
			\begin{equation}
				\label{eq:SymLocal}
				\SymD
				\alpha
				\coloneqq
				dz
				\vee
				(\nabla_z \alpha)
				+
				dw
				\vee
				(\nabla_w \alpha)
			\end{equation}
			for sections $\alpha \in \Gamma(S^k(T^*\Omega))$. This is indeed coordinate independent, linear and inherits the Leibniz rule
			\begin{equation}
				\label{eq:SymLeibniz}
				\SymD
				(\alpha \vee \beta)
				=
				(\SymD \alpha) \vee \beta
				+
				\alpha \vee (\SymD \beta)
				\, ,
				\qquad
				\alpha, \beta
				\in \Gamma\big(S(T^*\Omega)\big)
			\end{equation}
			for the symmetric tensor product $\vee$. Thus $\SymD \alpha \in \Gamma(S^{k+1}(T^*\Omega))$ is a well--defined section and the resulting map $\SymD$ is a linear derivation. All of our discussion may be restricted to open subsets of $\Omega$ such as $\Omega_\pm$ at once. For a much more comprehensive discussion of symmetrized covariant derivatives we refer to \cite[Appendix~A]{SchmittSchoetz2022}.

			\medskip

            To the best of our knowledge, the idea to use (essentially) symmetrized covariant derivatives to describe the Peschl-Minda differential operators appeared first in \cite{Minda}. In \cite[Ch.~3, p.~14]{Schippers1999}, the holomorphic part of $\SymD$ takes the form of as the partial connection $\tilde{\nabla}$, which moreover allows for different metrics for covariant and contravariant tensors. In \cite[Sec.~3]{KS07diff}, the same object effectively appears as the operator~$\Lambda$. Both note that $\SymD$ may be viewed as the full Levi-Civita covariant derivative $\nabla$ (i.e. without specifying a vector field, which raises the contravariant degree by one) followed by a projection onto the  $dz^n$ term.
                       
\medskip
            
            Using these variants of the symmetrized covariant derivative, one gets various equivalent expressions for the Peschl--Minda derivatives $D_z^n f$ for arbitrary maps $f \colon R \longrightarrow S$ between Riemann surfaces: \cite[(14.18)]{Nakahara2003} (only with target $S = \C$), \cite[Sec.~4.1, Def.~3]{Schippers1999}, \cite[(4.1)]{Schippers2007} and \cite[(3.5)]{KS07diff}. In our case, the metrics are given by $\sigma \equiv 1$ and $\varrho$ is the ``Riemannian length element'' $\lambda_\Omega(z,w) \coloneqq \tfrac{1}{1-zw}$ (see also the discussion \cite[p.23]{Schippers1999}). Note that we have $g_\Omega(z,w) = \lambda_\Omega^2 dz \vee dw$ by \eqref{eq:OmegaKaehler}.
            For our purposes, it is convenient to extract the coefficient functions by means of the natural pairing $\langle \argument, \argument \rangle$ between forms and multivectorfields, which yields the following incarnation of the aforementioned formulas.
			\begin{lemma}[Peschl--Minda operators vs. Symmetrized covariant derivatives]
				The pure Peschl--Minda derivatives may be expressed as
				\begin{equation}
					\label{eq:SymVsPM}
					D^{n}_z f
					=
					\big\langle
					\SymD^n f, (1-zw)^n \del^n_z
					\big\rangle
					\quad \text{and} \quad
					D^{n}_w f
					=
					\big\langle
					\SymD^n f, (1-zw)^n \del^n_w
					\big\rangle
				\end{equation}
				for $n \in \N_0$ and $f \in \mathcal{H}(\Omega \cap \C^2)$.
			\end{lemma}
			\begin{proof}
				Using the right hand side of \eqref{eq:SymVsPM}, we define the operators
				\begin{equation*}
					A_n f(z,w)
					\coloneqq
					\big\langle
					\SymD^n f, \lambda_\Omega^{-n} \del^n_z
					\big\rangle
					\, ,
					\qquad
					(z,w)
					\in
					\Omega \cap \C^2
				\end{equation*}
				for $f \in C^\infty(\Omega \cap \C^2)$ and $n \in \N_0$. Clearly, $A_n$ acts on $C^\infty(\Omega \cap \C^2)$ and $A_0$ is just the identity operator. Consequently, it suffices to show that the family $(A_n)$ obeys the recursion relation \eqref{eq:PMRecursionPure}. To this end, let $\Gamma_z \coloneqq dz(\nabla_{\! z} \del_z)$ be the non-trivial Christoffel symbol involving $z$.
				As our Levi-Civita connection is metric,
				i.e. $\nabla g_\Omega = 0$, we have
				\begin{equation*}
					0
					=
					\nabla_{\! \del_z}
					\big(
					\lambda_\Omega^2 dz \vee dw
					\big)
					=
					\bigg(
					2
					\frac{\partial \lambda_\Omega}{\partial z}
					\frac{1}{\lambda_\Omega}
					-
					\Gamma_z
					\bigg)
					g
					=
					\bigg(
					\frac{2w}{1-zw}
					-
					\Gamma_z
					\bigg)
					g \, ,
				\end{equation*}
				i.e. $\Gamma_z(z,w) = \tfrac{2w}{1-zw}$. Employing \eqref{eq:SymLocal}, \eqref{eq:CovPairing} and \eqref{eq:SymLeibniz}, we compute for $f \in C^\infty(\Omega \cap \C^2)$
				\begin{align*}
					(A_1 \circ A_{n}) f
					&=
					\Big\langle
					\SymD
					\big\langle
					\SymD^n f, \lambda_\Omega^{-n} \del_z^{n}
					\big\rangle,
					\lambda_\Omega^{-1} \del_z
					\Big\rangle \\
					&=
					\frac{1}{\lambda_\Omega}
					\nabla_z
					\big\langle
					\SymD^n f, \lambda_\Omega^{-n} \del_z^{n}
					\big\rangle \\
					&=
					\frac{1}{\lambda_\Omega}
					\Big\langle
					\nabla_z\big(\SymD^n f\big),
					\lambda_\Omega^{-n} \del_z^{n}
					\Big\rangle
					+
					\frac{1}{\lambda_\Omega}
					\Big\langle
					\SymD^n f, \nabla_z\big(\lambda_\Omega^{-n} \del_z^{n}\big)
					\Big\rangle \\
					&=
					\big\langle
					\SymD^{n+1} f, \lambda_\Omega^{-(n+1)} \del_z^{n+1}
					\big\rangle
					+
					\frac{n}{\lambda_\Omega^2}
					\bigg(
					-
					\frac{\partial \lambda_\Omega}{\partial z}
					+
					\Gamma_z \lambda_\Omega
					\bigg)
					\Big\langle
					\SymD^n f,
					\lambda_\Omega^{-n}
					\del^{n}_z
					\Big\rangle \\
					&=
					A_{n+1}f
					+
					nw
					A_n f \, ,
				\end{align*}
				which is \eqref{eq:PMRecursionPure} for $(A_n)$.
			\end{proof}

            Using \eqref{eq:SymVsPM}, one may give different proofs of some of our results for the pure Peschl--Minda differential operators. We confine ourselves to making a few remarks:
			\begin{mylist}
				\item[(a)] Taking another look at \eqref{eq:PMStandard} and \eqref{eq:SymVsPM}, it is tempting to conjecture
				\begin{equation*}
					D^{m,n}
					f
					=
					\big\langle
					\SymD^{m+n} f,
					\lambda_\Omega^{-(m+n)}
					\del_z^{m} \vee \del_w^n
					\big\rangle
				\end{equation*}
				for the mixed Peschl--Minda derivatives $D^{m,n}f$. However, this turns out to be \emph{wrong}: Already for $m=2$ and $n=1$ (but not for $m=n=1$) one gets an additional lower order contribution on the right--hand side. Using Theorem~\ref{thm:PMwithLaplace}, one may derive fairly involved formulas for the mixed Peschl--Minda differential operators $D^{m,n}$ generalizing \eqref{eq:SymVsPM}. The resulting complexity makes it impractical to use the abstract approach for the theory of Section~\ref{sec:PM}.
				\item[(b)] While the symmetrized covariant derivative turns out to be invariant, i.e.
				\begin{equation*}
					\SymD
					\circ
					T^*
					=
					T^*
					\circ
					\SymD
				\end{equation*}
				for the pullback $T^*$ with $T \in \mathcal{M}$, the vector field $\lambda_\Omega^{-1} \partial_z$ is not. Its transformation behaviour is what generates the prefactor in Proposition~\ref{prop:PMinvariance}.
				\item[(c)] If one goes through this construction using the flip chart \eqref{eq:ChartFlip} instead of the standard one, the resulting Peschl--Minda differential operators turn out to be $\tilde{D}^{0,n}$ and $\tilde{D}^{n,0}$ from \eqref{eq:PMFlip}.
			\end{mylist}

			Having expressed the pure Peschl--Minda differential operators in terms of symmetrized covariant derivatives, we next relate these operators to the star product $\star_{\hbar,\D}$ on the unit disk $\D$, which has been introduced in \cite{KrausRothSchoetzWaldmann2019} and was studied further in  \cite{SchmittSchoetz2022}. By a combination of Corollary~\ref{cor:PeschlMindaChangeOfCoordinates} and Theorem~\ref{thm:PMLinearisation} the following bi--differential operators are well--defined.
			\begin{definition}[Peschl--Minda Bi--Differential Operators] \label{def:BiDiff}
				Let $n \in \N_0$ and $U$ be an open subset of $\Omega$.
				For $f,g \in \mathcal{H}(U)$ and $(z,w) \in U$ we define
				\begin{equation*}
					B_n(f,g)
					\at[\Big]{(z,w)}
					\coloneqq
					\begin{cases}
						\big(
						D^{n}_z f
						\big)
						(z,w)
						\cdot
						\big(
						D^{n}_w g
						\big)(z,w)
						& \text{ if } (z,w)
						\in
						U \cap \C^2 \\[2mm]
						\big( D^{n}_w(f_-) \big)
						({\mathcal{F}(z,w)})
						\cdot
						\big(D^{n}_z(g_-)\big)
						({\mathcal{F}(z,w)})
						& \text{ if } (z,w)
						\in
						U \cap (\hat{\C}^* \times \hat{\C}^* )
					\end{cases}
				\end{equation*}
				and call
				\begin{equation}
					\label{eq:BiDiff}
					B_n
					\colon
					\mathcal{H}(U)
					\times
					\mathcal{H}(U)
					\longrightarrow
					\mathcal{H}(U)
				\end{equation}
				the Peschl--Minda bi--differential operator of order $n$.
			\end{definition}

			By Proposition~\ref{prop:PMinvariance} and Corollary~\ref{cor:PeschlMindaChangeOfCoordinates}, the Peschl--Minda bi--differential operators are $\mathcal{M}$--invariant, i.e.
			\begin{equation}
				\label{eq:PMBiInvariance}
				B_n
				\big(
				f \circ T,
				g \circ T
				\big)
				=
				B_n(f,g)
				\circ
				T
			\end{equation}
			holds for all $T \in \mathcal{M}$, $n \in \N_0$ and $f,g \in \mathcal{H}(\Omega)$. Furthermore, we may express the holomorphic Poisson structure as the antisymmetric part of the first Peschl--Minda bi--differential operator.
			\begin{lemma}[Holomorphic Poisson bracket]
				\label{lem:PoissonBracket}
				The holomorphic Poisson bracket $\{\argument, \argument\}_\Omega$ of the holomorphic Kähler manifold $(\Omega, g_\Omega)$ is given by
				\begin{equation}
					\label{eq:PoissonBracket}
					\{f, g\}_\Omega
					=
					B_1(f,g)
					-
					B_1(g,f) \, ,
					\qquad
					f,g \in \mathcal{H}(\Omega).
				\end{equation}
			\end{lemma}
			\begin{proof}
				It suffices to check \eqref{eq:PoissonBracket} for $(z,w) \in \Omega \cap \C^2$, i.e. in the standard chart. In these coordinates the symplectic form $\omega_\Omega$ takes the form
				\begin{equation*}
					\omega_\Omega(z,w)
					=
					\frac{dz \wedge dw}{(1-zw)^2}
				\end{equation*}
				by virtue of \eqref{eq:OmegaKaehler}. Consequently, the corresponding Hamiltonian vector fields are
				\begin{equation*}
					X_f
					=
					(1-zw)^2
					\bigg(
					\frac{\partial f}{\partial w}
					\partial_z
					-
					\frac{\partial f}{\partial z}
					\partial_w
					\bigg)
					\, , \qquad
					f \in \mathcal{H}(\Omega).
				\end{equation*}
				And so the Poisson bracket takes the form
				\begin{equation*}
					\{f, g\}_\Omega
					=
					X_g(f)
					=
					(1-zw)^2
					\Big(
					\frac{\partial f}{\partial z}
					\frac{\partial g}{\partial w}
					-
					\frac{\partial f}{\partial w}
					\frac{\partial g}{\partial z}
					\Big)
					=
					B_1(f,g)
					-
					B_1(g,f)
				\end{equation*}
				for $f,g \in \mathcal{H}(\Omega)$, which proves \eqref{eq:PoissonBracket}.
			\end{proof}

			Preparing for various continuity statements, we note the following lemma, which ultimately boils down to an application of the classical Cauchy estimates to the Peschl--Minda derivatives. 
			Given a compact subset $K \subseteq \Omega$, we write $\norm{f}_K \coloneqq \max_{(z,w) \in K} \abs{f(z,w)}$ for continuous functions $f \colon K \longrightarrow \C$.
			\begin{lemma}[Cauchy estimates for the Peschl--Minda derivatives]
				\label{lem:PMCauchy}%
				Let $f \in \mathcal{H}(\Omega_+)$, $g \in \mathcal{H}(\Omega_-)$ and $K_\pm \subseteq \Omega_\pm \cap \C^2$ be compact sets. Then for every $R > 0$ there exist compact sets $L_\pm \subseteq \Omega_\pm$ such that
				\begin{equation}
					\label{eq:PMCauchy}
					{\norm[\big]
					{D^n_z f}}_{K_+}
					\le
					\frac{n!}{R^n}
					{\norm{f}}_{L_+}
					\quad \text{and} \quad
					{\norm[\big]
					{D^n_w g}}_{K_-}
					\le
					\frac{n!}{R^n}
					{\norm{g}}_{L_-}
					\, , \qquad
					n \in \N_0 \, .
				\end{equation}
				Moreover, for every compact set $K \subseteq \Omega$ and every $R > 0$, there exist compact sets $L_1, L_2 \subseteq \Omega$ with
				\begin{equation}
					\label{eq:PMCauchyBi}
					\norm[\big]
					{B_n(f,g)}_{K}
					\le
					\frac{\left(n!\right)^2}{R^{2n}}
					{\norm{f}}_{L_1}
					{\norm{g}}_{L_2}
					\, , \qquad
					n \in \N_0 .
				\end{equation}
			\end{lemma}
			\begin{proof}
				The functions $u \mapsto (f \circ \Phi_{z,w})(u,0)$ resp. $v \mapsto (g \circ \Phi_{z,w})(0,v)$ are entire. Hence the Cauchy estimates imply for every $R > 0$ and $n \in \N_0$
				\begin{equation*}
					\norm[\big]
					{D^n_z f}_{K_+}
					=
					\max_{(z,w) \in K_+}
					\abs[\big]
					{
						\partial_1^n
						\big(f \circ \Phi_{z,w}\big)
						(0,0)
					}
					\le
					\max_{(z,w) \in K_+}
					\frac{n!}{R^n}
					\max_{\abs{u} = R}
					\big|
					\big( f \circ \Phi_{z,w} \big)(u,0) \big|
					=
					\frac{n!}{R^n}
					{\norm{f}}_{L_+}
				\end{equation*}
				with
				\begin{equation*}
					L_+
					\coloneqq
					\Big\{
					\Big(
					\frac{z+u}{1+uw}, w
					\Big)
					\colon
					(z,w)
					\in K_+,
					\abs{u} = R
					\Big\}
					\subseteq \Omega_+
					\, .
				\end{equation*}

				The inequality \eqref{eq:PMCauchyBi} may be established essentially in the same manner, but we have to be careful about the points near infinity. We define the compacta $K_0 \coloneqq \{(z,w) \in K \colon \abs{zw} \le 1\}$ and $K_\infty \coloneqq \{(z,w) \in K \colon \abs{zw} \ge 1\}$, where we set $\abs{\infty} \coloneqq \infty$ and use the extended arithmetic for the product $zw$ as before. By construction, $K_0 \cup K_\infty = K$, $K_0 \subseteq \C^2$ and $K_\infty \subseteq (\hat{\C}^*)^2$. On $K_0$, the first part gives estimates of the form \eqref{eq:PMCauchy}. On $K_\infty$, the Peschl--Minda bi--differential operator is given by $B_n(f,g) = \tilde{D}_{z}^n f \cdot \tilde{D}_{w}^n g$. We observe that also the compositions
                \begin{equation*}
                    u
                    \mapsto
                    (f \circ \phi_-^{-1} \circ \Phi_{1/w,1/z})(u,0)
                    \quad \text{resp.} \quad
                    v
                    \mapsto
                    (g \circ \phi_-^{-1} \circ \Phi_{1/w,1/z})(0,v)
                \end{equation*}
            are entire functions. As before, the Cauchy estimates provide the desired inequality on $K_\infty$. Putting everything together, we arrive at \eqref{eq:PMCauchyBi}.
			\end{proof}

			Implicitly, the Wick star product on $\Omega$ already appears as an intermediate step in the construction of the Wick star product $\star_{\hbar,\D}$ in \cite{KrausRothSchoetzWaldmann2019}. In particular, the observable algebra
			\begin{equation*}
				\mathcal{A}(\D)
				\coloneqq
				\big\{
				d_\D^*
				f
				=
				f \circ d_\D
				\colon
				\D \longrightarrow \C
				\;\big|\;
				f \in \mathcal{H}(\Omega)
				\big\}
			\end{equation*}
			is defined as the pullback of $\mathcal{H}(\Omega)$ by the diagonal map $d_\D$, see \eqref{eq:Diagonal}, and convergence of a sequence $(d_\D^* f_n)$ in $\mathcal{A}(\D)$ is defined as locally uniform convergence of the holomorphic extensions $(f_n)\subseteq \mathcal{H}(\Omega)$. Consequently, it is more natural to look at $\mathcal{H}(\Omega)$ directly. This has the additional technical benefit of working with holomorphic functions instead of just real analytic ones. To arrive at an explicit formula for the star product on $\mathcal{H}(\Omega)$, our strategy is to express the star product on $\mathcal{A}(\D)$ as defined in \cite{KrausRothSchoetzWaldmann2019} and \cite{SchmittSchoetz2022} by means of the one-variable Peschl--Minda operators acting on smooth functions.	  In a second step, we then `lift'' the star product from $\mathcal{A}(\D)$ to $\mathcal{H}(\Omega)$ by identifying all relevant objects as restrictions of objects which we have already defined on $\mathcal{H}(\Omega)$ in the previous Sections~\ref{sec:PM}--\ref{sec:PMClassical}. This also shows that the star product on $\mathcal{H}(\Omega)$ inherits the algebraic properties of $\star_{\hbar,\D}$ such as bilinearity and associativity.

			\medskip

			It turns out that the combinatorics in the Wick star product may be expressed using the Peschl--Minda derivatives and (falling) Pochhammer symbols, which we denote by
			\begin{equation*}
				\label{eq:Pochhammer}
				(z)_{n \downarrow}
				\coloneqq
				z(z-1)\cdots(z-n+1)
				=
				\prod_{j=0}^{n-1}
				(z - j)
				\, ,
				\qquad
				z \in \C
				\, , \;
				n \in \N_0.
			\end{equation*}
			The case $(z)_{n \downarrow} = 0$ only occurs for $z \in \N_0$ and $n > z$. For our purposes only Pochhammer symbols of the form $(-1/\hbar)_{n \downarrow}$ for parameters $\hbar \in \C$ which belong to the so--called \emph{deformation domain}
				$$ \mathscr{D}:=\C^*\setminus \left\{ -\frac{1}{n} \, : \, n \in \N \right\} $$
				are relevant. In particular, we then always have $(-1/\hbar)_{n \downarrow}\not=0$.
\medskip
    
            We are now in a position to provide an explicit formula for the Wick star product $\star_{\hbar,\D}$ on $\mathcal{A}(\D)$, which was developed in \cite{KrausRothSchoetzWaldmann2019} and \cite{SchmittSchoetz2022}, in terms of Peschl--Minda derivatives.

			\begin{lemma}[Wick star product on $\mathcal{A}(\D)$]
                \label{lem:StarProductDisk}
				Let $\hbar \in \mathscr{D}$ and $\varphi, \eta \in \mathcal{A}(\D)$. Then
				\begin{equation}
					\label{eq:StarProductDisk}
					\varphi \star_{\hbar,\D} \eta
					=
					\sum_{n=0}^\infty
					\frac{(-1)^n}{n!}
					\frac{1}{(-1/\hbar)_{n\downarrow}}
					\big(\cc{D}^n \varphi \big)
					\cdot
					\big(D^n \eta \big)
				\end{equation}
				with the Peschl--Minda differential operators $D^n$ and $\cc{D}^n$ acting on $C^\infty(\D)$. Moreover,
				the series \eqref{eq:StarProductDisk} converges in $\mathcal{A}(\D)$.
			\end{lemma}
			\begin{proof}
				The key observation is that the reduced Hamiltonian $H \in \Gamma^\infty(S^2(T\D))$ from \cite[(5.19)]{SchmittSchoetz2022} simplifies in dimension one: all the summations collapse to a single term and the signature is $\nu = -1$, which implies
				\begin{equation*}
					H(z)
					=
					\nu
					(1-z\cc{z})^2
					\cc{\partial} \otimes \partial
					=
					-
					\frac{\cc{\partial}}{\lambda_\D(z)}
					\otimes
					\frac{\partial}{\lambda_\D(z)} \, ,
					\qquad
					z \in \D
					\, ,
				\end{equation*}
				where $\lambda_\D(z) \coloneqq \lambda_\Omega(z,\cc{z})$ is the usual hyperbolic length element. Hence, we are in a position to apply one of the central results of \cite{SchmittSchoetz2022}, namely their formula (5.24). Note that we have chosen~a different sign convention for $\hbar$, so we need to apply (5.24) in \cite{SchmittSchoetz2022} for $-\hbar$ instead of $\hbar$. This way, we arrive at 
				\begin{align*}
					\varphi \star_{\hbar,\D} \eta
					&=
					\sum_{n=0}^\infty
					\frac{1}{n!}
					\frac{1}{(-1/\hbar)_{n\downarrow}}
					\big\langle
					\SymD^{n} \varphi \otimes \SymD^n \eta, H^n
					\big\rangle \\
					&=
					\sum_{n=0}^\infty
					\frac{(-1)^n}{n!}
					\frac{1}{(-1/\hbar)_{n\downarrow}}
					\big\langle
					\SymD^{n} \varphi, \lambda_\D^{-n} \delbar^n
					\big\rangle
					\big\langle
					\SymD^{n} \eta, \lambda_\D^{-n} \del^n
					\big\rangle\\
					&=
					\sum_{n=0}^\infty
					\frac{(-1)^n}{n!}
					\frac{1}{(-1/\hbar)_{n\downarrow}}
					\big(\cc{D}^n \varphi \big)
					\cdot
					\big(D^n \eta \big) \, ,
				\end{align*}
				where in the last step we have applied \eqref{eq:SymVsPM} with $w = \cc{z}$.
                Here, the first equality sign is formula (5.24) in \cite{SchmittSchoetz2022} for $\star_{\hbar,\D}$ with $\hbar$ replaced by $-\hbar$ as explained above.
			\end{proof}

			Equipped with these preliminaries, we can now state the main result of this section.

			\begin{theorem}[Wick star product on $\mathcal{H}(\Omega)$]
				\label{thm:StarProduct}
				Let $B_n$ be the Peschl--Minda bi--differential operators from \eqref{eq:BiDiff} and $f,g \in \mathcal{H}(\Omega)$. Then the factorial series
				\begin{equation}
					\label{eq:StarProduct}
					f \star_\hbar g
					\coloneqq
					\sum_{n=0}^\infty
					\frac{(-1)^n}{n!}
					\frac{1}{(-1/\hbar)_{n\downarrow}}
					B_n(g,f)
				\end{equation}
				converges absolutely and locally uniformly on $\Omega$ as well as locally uniformly w.r.t.~$\hbar \in \mathscr{D}$. 
    In particular, the mapping
				\begin{equation}
					\label{eq:StarProductHolomorphic}
					\mathscr{D}
					\ni
					\hbar
					\mapsto
					f \star_\hbar g
					\in
					\mathcal{H}(\Omega)
				\end{equation}
				is holomorphic. For every $\hbar \in \mathscr{D}$, the triple $(\mathcal{H}(\Omega), +, \star_\hbar)$ is a Fréchet algebra with respect to the topology of locally uniform convergence on $\Omega$.
			\end{theorem}

			We call the product $\star_\hbar$ the Wick star product on $\mathcal{H}(\Omega)$.
			\begin{proof} Note that each of the partial sums in \eqref{eq:StarProduct} can be viewed as a holomorphic function on $\mathscr{D} \times \Omega$, since $\hbar \mapsto (-1/\hbar)_{n\downarrow}$ is holomorphic on $\mathscr{D}$ and $B_n(g,f) \in \mathcal{H}(\Omega)$.
				We begin by proving that the series \eqref{eq:StarProduct} converges absolutely in $\mathcal{H}(\mathscr{D} \times \Omega)$.
				Let $\hat{K}$ be~a compact subset of $\mathscr{D}$. Then $\hat{K}$ has positive distance to the boundary of the deformation domain $\mathscr{D}$, hence there exists an $\alpha > 0$ such that
				\begin{equation}
					\label{eq:PochhammerEstimate}
					\abs[\big]
					{(-1/\hbar)_{n \downarrow}}
					\ge
					\frac{n!}{\alpha^n}
					\, , \qquad
					\hbar \in \hat{K}
					\, , n \in \N_0.
				\end{equation}
				Let $\hbar \in \hat{K}$ and $f,g \in \mathcal{H}(\Omega)$. If $R > \sqrt{\alpha}$ and $K \subseteq \Omega$ is a compact set, then \eqref{eq:PMCauchyBi} yields
				\begin{equation}
					\label{eq:StarProductContinuityEstimate}
					\sum_{n=0}^\infty
					\frac{1}{n!}
					\frac{1}{\abs{(-1/\hbar)_{n\downarrow}}}
					\norm[\big]{B_n(g,f)}_K
					\le
					{\norm{g}}_{L_1}
					{\norm{f}}_{L_2}
					\sum_{n=0}^\infty
					\frac{\alpha^n}{R^{2n}}
					=:
					C \cdot
					{\norm{g}}_{L_1}
					{\norm{f}}_{L_2}
				\end{equation}
				with the constant $C > 0$ only depending on $\alpha$ and $R,$ and therefore, depending on appropriately chosen compact sets $L_1, L_2 \subseteq \Omega$. This implies the absolute convergence of the series \eqref{eq:StarProduct} in $\mathcal{H}(\mathscr{D} \times \Omega)$.  By completeness of $\mathcal{H}(\mathscr{D} \times \Omega)$, the function
                                $$ (\hbar,(z,w)) \mapsto (f \star_\hbar g)(z,w)$$ is thus well--defined and holomorphic on $\mathscr{D} \times \Omega$. Another consequence of \eqref{eq:StarProductContinuityEstimate} is that the series \eqref{eq:StarProduct} converges absolutely in $\mathcal{H}(\Omega)$ for every $\hbar \in \mathscr{D}$ and thus $f \star_\hbar g \in \mathcal{H}(\Omega)$. If $\varphi \colon \mathcal{H}(\Omega) \to \C$ is a continuous linear functional, we have
				\begin{equation*}
					\varphi(f \star_\hbar g)
					=
					\sum_{n=0}^\infty
					\frac{(-1)^n}{n!}
					\frac{1}{(-1/\hbar)_{n\downarrow}}
					\varphi\big(B_n(g,f)\big) \,
					\qquad
					\hbar \in \mathscr{D}.
				\end{equation*}
				As each of the partial sums is holomorphic with respect to $\hbar$, this implies holomorphy of the mapping $\mathscr{D} \ni \hbar \mapsto \varphi(f \star_\hbar g)$. That is, $\mathscr{D} \ni \hbar \mapsto f \star_\hbar g \in \mathcal{H}(\Omega)$ is weakly holomorphic and by \eqref{eq:StarProductContinuityEstimate} also locally bounded. By \cite[Prop.~3.7]{Dineen} this implies (Fréchet) holomorphy of \eqref{eq:StarProductHolomorphic}.
        Recall now that on $\Omega \cap \C^2$, we have $(D^n_w f) (D^n_z g) = B_n(g,f)$. Therefore, \eqref{eq:StarProduct} extends the Wick star product on $\D$ as expressed by \eqref{eq:StarProductDisk} to the star product $\star_\hbar$ on $\Omega$. In particular, $\star_\hbar$ is a multiplication on $\mathcal{H}(\Omega)$ for every $\hbar \in \mathscr{D}$ by the aforementioned identity principle \cite[p.~18]{Range}. Finally, \eqref{eq:StarProductContinuityEstimate} yields
				\begin{equation*}
					{\norm[\big]
					{f \star_\hbar g}}_K
					\le
					C
					\cdot
					{\norm{g}}_{L_1}
					\cdot
					{\norm{f}}_{L_2}
					\, .
				\end{equation*}
				Hence, the bilinear mapping $\star_\hbar \colon \mathcal{H}(\Omega) \times \mathcal{H}(\Omega) \longrightarrow \mathcal{H}(\Omega)$ is continuous. This shows that $(\mathcal{H}(\Omega),+,\star_{\hbar})$ is a Fr\'echet algebra.
			\end{proof}

			By continuity of pullbacks with automorphisms $T \in \mathcal{M}$ and \eqref{eq:PMBiInvariance}, we deduce that the Wick star product on $\mathcal{H}(\Omega)$ is invariant under the \textit{full} M\"obius--type group $\mathcal{M}$:
			\begin{corollary}[$\mathcal{M}$-invariance of the Wick star product]
				\label{cor:StarProductInvariance}
				The Wick star product $\star_\hbar$ is $\mathcal{M}$-invariant, i.e.
				\begin{equation}
					\label{eq:StarProductInvariance}
					(f \circ T) \star_\hbar (g \circ T)
					=
					(f \star_\hbar g)
					\circ
					T
				\end{equation}
				for $f, g \in \mathcal{H}(\Omega)$, $T \in \mathcal{M}$ and $\hbar \in \mathscr{D}$.
			\end{corollary}

  The next goal is to derive an asymptotic expansion for the star product on $\mathcal{H}(\Omega)$. It will turn out that the so--called
 \textit{Stirling numbers of the second kind} denoted by $$ \begin{Bmatrix} n \\k \end{Bmatrix} \, $$
 play an essential role. For $k,n \in \N_0$ these numbers are defined by the identity (\cite[24.1.4 B]{abramowitz1984})
 $$ \sum \limits_{k=0}^n \begin{Bmatrix} n \\ k \end{Bmatrix} {(x)}_{k\downarrow}=x^n \, ,
 $$ and it is moreover convenient to define
 $$  \begin{Bmatrix} 0 \\ -1 \end{Bmatrix} :=1 \quad \text{  and }  \quad \begin{Bmatrix} n \\ -1 \end{Bmatrix}:= 0\, , \qquad n \in \N  \, .$$
 We note that
 $$ \begin{Bmatrix} 0 \\ 0\end{Bmatrix}=1  \quad \text{ and } \quad \begin{Bmatrix} n-1 \\ 0 \end{Bmatrix}= 0  \,  \, , \qquad n \ge 2 \, , $$
 and
 \begin{equation*}
0 \le \begin{Bmatrix} n-1 \\ k-1 \end{Bmatrix} \le \frac{1}{2} \binom{n-1}{k-1} (k-1)^{n-k} \, , \qquad
    2 \le k \le n-1 \, ,
 \end{equation*}
 see \cite{RennieDobson}. Therefore,
  \begin{equation}\label{eq:EstimateStirling}
   0 \le  \begin{Bmatrix} n-1 \\ k-1 \end{Bmatrix}  \le  n^k k^n \qquad \text{ for all } 1 \le k \le n \, , \quad n \in \N \, .
    \end{equation}
This rather crude estimate suffices for our purposes. Finally, the connection of the Stirling numbers of the second kind to the Pochhammer symbol is given by

\begin{equation}
            \label{eq:PochhammerPowerSeries}
            \frac{1}{(-1/z)_{k \downarrow}}
            =
            \sum_{n=k}^\infty
             (-1)^n \begin{Bmatrix} n-1 \\ k-1 \end{Bmatrix}
             z^n \, ,
            \qquad
            \abs{z} < \frac{1}{k-1}
            \, ,
        \end{equation}
for every $k \in \N$, see \cite[no.~24.1.4 (b)]{abramowitz1984}. If $k=1$, then the expansion (\ref{eq:PochhammerPowerSeries}) holds for all $z \in \C$.

    \begin{theorem}[Asymptotic expansion]
        \label{thm:StarProductAsymptotics}
        Let $f,g \in \mathcal{H}(\Omega)$, $K \subseteq \Omega$ be compact and $N \in \N_0$. Moreover, let $\eps \in (0,\pi]$. Then
        \begin{equation}
            \label{eq:StarProductAsymptotic}
            {\norm[\bigg]
            {
                f \star_\hbar g
                -
                \sum_{n=0}^N
                \hbar^n
                \sum_{k=0}^n
                \frac{(-1)^{k+n}}{k!}
                \begin{Bmatrix} n-1 \\ k-1 \end{Bmatrix}
                B_k(g,f)
            }}_K
            =
            \mathcal{O}\left(\hbar^{N+1}\right)
        \end{equation}
        uniformly as $\hbar \to 0+$ through the sector $S_{\eps}:=\{\hbar \in \C\setminus\{0\} \, :|\arg(\hbar)| \le \pi-\eps\}$.
         \end{theorem}
    \begin{proof} We write $$\beta_{n,k}:=(-1)^n \begin{Bmatrix} n-1 \\ k-1 \end{Bmatrix}$$ for short.
        Let $\hbar \in \mathscr{D}$. We split the convergent series $f \star_\hbar g$ (see Theorem \ref{thm:StarProduct}) into
        \begin{equation*}
            F_1(\hbar)
            \coloneqq
            \sum_{k=0}^N
            \frac{(-1)^k}{k!}
            \frac{1}{(-1/\hbar)_{k \downarrow}}
            B_k(g,f)
            \quad \text{and} \quad
            F_2(\hbar)
            \coloneqq
            \sum_{k=N+1}^\infty
            \frac{(-1)^k}{k!}
            \frac{1}{(-1/\hbar)_{k \downarrow}}
            B_k(g,f) \, .
        \end{equation*}
        By \eqref{eq:PochhammerPowerSeries}, we have
        \begin{align*}
            F_1(\hbar)
            &=
            \sum_{k=0}^N
            \frac{(-1)^k}{k!}
            \sum_{n=k}^\infty
            \hbar^n
            \beta_{n,k}B_k(g,f) \\
            &=
            \sum_{n=0}^\infty
            \hbar^n
            \sum_{k=0}^{\min\{n,N\}}
            \frac{(-1)^k}{k!}
            \beta_{n,k}
            B_k(g,f) \\
            &=
            \sum_{n=0}^N
            \hbar^n
            \sum_{k=0}^{n}
            \frac{(-1)^k}{k!}
            \beta_{n,k}
            B_k(g,f)
            +
            \sum_{n=N+1}^\infty
            \hbar^n
            \sum_{k=0}^{N}
            \frac{(-1)^k}{k!}
            \beta_{n,k}
            B_k(g,f)
        \end{align*}
       for $\hbar$ belonging to the open disk around $0$ with radius $1/(N-1)$.
        We recognize the first term as what we subtract in the asymptotics \eqref{eq:StarProductAsymptotic}, and it therefore remains to estimate the second term and $F_2$ in modulus.
       We first estimate the second term.
       Assuming $\abs{\hbar} < 1/(N+1)$, using the Cauchy--type estimate \eqref{eq:PMCauchyBi} with $R = 1$, and the bound \eqref{eq:EstimateStirling} for the Stirling numbers of the second kind, we see that there is a constant $M_K$ depending only on $f$, $g$ and $K$ such that
        \begin{align*}
            \norm[\bigg]
            {
                \sum_{n=N+1}^\infty
                \hbar^n
                \sum_{k=0}^{N}
                \frac{(-1)^k}{k!}
                \beta_{n,k}
                B_k(g,f)
            }_K
            &\le
            \abs{\hbar}^{N+1}
            \sum_{n=N+1}^\infty
            \abs{\hbar}^{n-N-1}
            \sum_{k=0}^{N}
            \frac{n^k \cdot k^n}{k!}
            M_K \left(k!\right)^2 \\
            &\le
            \abs{\hbar}^{N+1}
            M_K N!
            \sum_{n=N+1}^\infty
            \abs{\hbar}^{n-N-1}
            n^N N^n \\
            &\le
            \abs{\hbar}^{N+1}
            M_K N!
            \sum_{n=N+1}^\infty
            \bigg(
                \frac{N}{N+1}
            \bigg)^n
            n^N
            (N+1)^{N+1}  \\
            &=:
            \abs{\hbar}^{N+1} C_N \, ,
        \end{align*}
        where $C_N$ depends only on $f,g$, $K$ and $N$. In particular,  $C_N$ does not depend on $\hbar$.
\end{proof}

    We see that asymptotically  both the so--called classical and the semiclassical limits exist in \eqref{eq:StarProduct}: indeed, within  any fixed sector $S_{\eps}$, $\eps\in(0,\pi]$, we have
    $f \star_{\hbar} g \to f g$ as $\hbar \to 0$ as well as

		\begin{equation*}
				\frac{1}{\hbar}
				\Big(
				f \star_\hbar g
				-
				g \star_\hbar f
				\Big)
				\overset{\hbar \rightarrow 0}{\longrightarrow}
				-
				\big(
				B_1(g,f) - B_1(f,g)
				\big)
				=
				{\{f,g\}}_\Omega
			\end{equation*}
			for $f,g \in \mathcal{H}(\Omega)$ by virtue of \eqref{eq:PoissonBracket}. Thus Theorem~\ref{thm:StarProductAsymptotics} generalizes \cite[Thm.~4.5]{KrausRothSchoetzWaldmann2019}, which asserts that the functions
			\begin{equation*}
				[0,\infty)
				\ni
				\hbar
				\quad \mapsto \quad
				\begin{cases}
					\varphi \star_{\hbar,\D} \eta
					\; &\text{for} \;
					\hbar > 0 \\
					\varphi \cdot \eta
					\; &\text{for} \;
					\hbar = 0
				\end{cases}
			\end{equation*}
			and
			\begin{equation*}
				[0,\infty)
				\ni
				\hbar
				\quad \mapsto \quad
				\begin{cases}
					\frac{1}{\hbar}
					\big(
					\varphi \star_{\hbar,\D} \eta
					-
					\eta \star_{\hbar,\D} \varphi
					\big)
					\; &\text{for} \;
					\hbar > 0 \\
					\{\varphi,\eta\}_\Omega
					\; &\text{for} \;
					\hbar = 0
				\end{cases}
			\end{equation*}
	are continuous for fixed $\varphi, \eta \in \mathcal{A}(\D)$. The proof in \cite{KrausRothSchoetzWaldmann2019} of this special case of Theorem \ref{thm:StarProductAsymptotics} is much more complicated than the proof of Theorem \ref{thm:StarProductAsymptotics} which in addition allows that $\hbar$ approaches $0$ through any sector $S_{\eps}$ and not only through the positive reals. Theorem \ref{thm:StarProductAsymptotics} shows that  one may think of the Fréchet algebra $(\mathcal{H}(\Omega), \star_\hbar)$ as a deformation of the holomorphic Poisson algebra $(\mathcal{H}(\Omega),\cdot,{\{\argument, \argument\}}_\Omega)$ with the Poisson bracket from \eqref{eq:PoissonBracket}.

            \medskip

            Using the diagonal mapping $d_{\hat{\C}}$, we recover the Wick star product on $\hat{\C}$, which was studied in \cite{EspositoSchmittWaldmann2019}, by means of pullbacks. More precisely, we have
			\begin{equation*}
				(d_{\hat{\C}}^* f) \star_{\hbar,\hat{\C}}
				(d_{\hat{\C}}^* g)
				\coloneqq
				d_{\hat{\C}}^*
				\big(
				f \star_\hbar g
				\big)
			\end{equation*}
	for $f, g \in \mathcal{H}(\Omega)$. This yields a multiplication on $\mathcal{A}(\hat{\C}) \coloneqq \{d_{\hat{\C}}^* f \colon f \in \mathcal{H}(\Omega)\}$. Convergence of a sequence $(d_{\hat{\C}}^* f_n)$ in $\mathcal{A}(\hat{\C})$ is defined as locally uniform convergence of the holomorphic extensions $(f_n)\subseteq \mathcal{H}(\Omega)$. By Theorem~\ref{thm:StarProduct}, we have shown the following.
    \begin{corollary}[Wick product on $\mathcal{A}(\hat{\C})$] \label{cor:StarProductSphere}
        Let $\hbar \in \mathscr{D}$ and $f,g \in \mathcal{H}(\Omega)$. Then the Wick star product on $\hat{\C}$ is given by
        \begin{equation*}
            (d_{\hat{\C}}^* f) \star_{\hbar,\hat{\C}}
            (d_{\hat{\C}}^* g)
            =
            \sum_{n=0}^\infty
            \frac{1}{n!}
            \frac{1}{(-1/\hbar)_{n\downarrow}}
            d_{\hat{\C}}^*
            \big(B_n(f,g)\big)
            \, .
        \end{equation*}
        and the series converges in $\mathcal{A}(\hat{\C})$ locally uniformly with respect to $\hbar$.
        In particular, the mapping
		\begin{equation*}
		      \mathscr{D}
			\ni
			\hbar
			\mapsto
			(d_{\hat{\C}}^* f) \star_{\hbar,\hat{\C}}
            (d_{\hat{\C}}^* g)
			\in
			\mathcal{A}(\hat{\C})
		\end{equation*}
		is holomorphic and $(\mathcal{A}(\hat{\C}), \star_{\hbar,\hat{\C}})$ is a Fréchet algebra.
    \end{corollary}
    \begin{corollary}
		Let $\varphi, \eta \in \mathcal{A}(\D)$. Then the factorial series \eqref{eq:StarProductDisk} converges in $\mathcal{A}(\D)$ locally uniformly with respect to $\hbar$. In particular, the mapping
				\begin{equation*}
					\mathscr{D}
					\ni
					\hbar
					\mapsto
					\varphi \star_{\hbar,\D} \eta
					\in
					\mathcal{A}(\D)
				\end{equation*}
				is holomorphic.
			\end{corollary}

        We conclude the section with some remarks.
        \begin{remark}
        \begin{mylist}
            \item[(a)] In \cite{SchmittSchoetz2022} an isomorphism between the Fréchet algebras $(\mathcal{A}(\D), \star_{\hbar,\D})$ and $(\mathcal{A}(\hat{\C}), \star_{\hbar,\hat{\C}})$ was constructed, the so--called \emph{Wick rotation}. From the point of view of $\Omega$, the Wick rotation corresponds to extending $d_\D^* f \in \mathcal{A}(\D)$ to $f \in \mathcal{H}(\Omega)$ and then restricting to $d_{\hat{\C}}^* f \in \mathcal{A}(\hat{\C})$ on the rotated diagonal. In passing, we note that the different signature $\nu$ in \cite[(5.24)] {SchmittSchoetz2022} in the cases $\D$ and $\hat{\C}$ is thus a consequence of the chain rule.
            \item[(b)] The invariance \eqref{eq:StarProductInvariance} reduces to invariance under $\Aut(\D)$ and the rotations of the sphere, i.e. the respective automorphism groups.
            \item[(c)] The asymptotic expansion from Theorem~\ref{thm:StarProductAsymptotics} may be restricted to the diagonal and rotated diagonal, yielding asymptotic expansions of $\star_{\hbar,\D}$ and $\star_{\hbar,\hat{\C}}$.
            \item[(d)] By \cite[Sec.~3.2]{HeinsMouchaRoth3}, the complex stereographic projection $$S \colon \Omega \longrightarrow \mathbb{S}_\C^2 = \{(z_1,z_2,z_3) \in \C^3 \colon z_1^2 + z_2^2 + z_3^2 = 1\}$$ is biholomorphic. Hence we may define
	   \begin{equation*}
	       \varphi
			\star_{\hbar, \mathbb{S}_\C^2}
			\eta
						\at[\Big]{(z_1, z_2, z_3)}
						\coloneqq
						(S^* \varphi)
						\star_\hbar
						(S^* \eta)
						\at[\Big]{S^{-1}(z_1, z_2, z_3)}
	\end{equation*}
	for $\varphi, \eta \in \mathcal{H}(\mathbb{S}_\C^2)$ and $(z_1, z_2, z_3) \in \mathbb{S}^2_\C$. Moreover, the pullback of Riemannian metric \eqref{eq:OmegaKaehler} along $S^{-1}$ coincides with the complexified spherical metric on the complex two--sphere $\mathbb{S}_\C^2$. By \cite[Sec.~4]{GuilleminStenzel}, $\mathbb{S}_\C^2$ is isomorphic to the cotangent bundle $T^* \mathbb{S}^2_\R$ of the real two--sphere $$\mathbb{S}^2_\R \coloneqq \big\{(x_1, x_2, x_3) \in \R^3 \colon x_1^2 + x_2^2 + x_3^2 = 1\big\}$$ as a Kähler manifold. Here, we endow the cotangent bundle $T^* \mathbb{S}^2_\R$ with the symplectic form induced by the tautological one--form, see \cite[Sec.~2]{CannasDaSilva}. Consequently, one may regard the star product as a quantization of $T^* \mathbb{S}^2_\R$. In \cite{HeinsRothWaldmann} a strict deformation of the cotangent bundle of any Lie group was constructed. In particular, this yields a quantization of the cotangent bundle $T^* \mathbb{S}^3_\R$ of the quaternion group. As $\mathbb{S}^2_\R \cong \mathbb{S}^3_\R / \mathbb{S}^1_\R$ may be regarded as a symmetric space, one should be able to relate these quantizations by a quantum version of cotangent bundle reduction. Finally, we note that in \cite{HallMitchell2001} another quantization of $T^* \mathbb{S}^2_\R$ was constructed. It would be interesting to compare both approaches.
\end{mylist}
\end{remark}

	\section{Continuous module structures induced by the Wick Star Product}
	The idea of this section is that we only need to control one of the functions $f$ and $g$ in \eqref{eq:StarProduct} to prove the absolute convergence of the series. This is reflected in allowing for differing domains of definition for $f$ and $g$. However, the Peschl--Minda bi--differential operators $B_n(g,f)$ and thus the series in \eqref{eq:StarProduct} are only defined in $(z,w) \in \Omega$ whenever both $f$ and $g$ are. This may be encoded as a (bi)module structure. We slightly abuse notation and denote the module multiplications induced by the series in \eqref{eq:StarProduct} by the same symbol $\star_\hbar$ as the star product. Recall that $\D^2$ is contained in each of the three domains $\Omega$, $\Omega_+$ and $\Omega_-$, see again Figure~\ref{fig:OmegaPM}.
	\begin{proposition}[Module structures I]
	\label{prop:StarProductBimoduleDisk}
    Let $\hbar \in \mathscr{D}$. The bilinear mappings
	\begin{equation}
        \label{eq:ModuleOne}
	    \star_\hbar
	    \colon
		\mathcal{H}(\D^2)
		\times
		\mathcal{H}(\Omega_+)
		\longrightarrow
		\mathcal{H}(\D^2)
		\quad \text{and} \quad
		\star_\hbar
		\colon
		\mathcal{H}(\Omega_-)
		\times
		\mathcal{H}(\D^2)
		\longrightarrow
		\mathcal{H}(\D^2)
	\end{equation}
	are well--defined, continuous and compatible in the sense that
    \begin{equation}
        \label{eq:BimoduleCompatibility}
        \big(
            g_-
            \star_\hbar
            f
        \big)
        \star_\hbar
        g_+
        =
        g_-
        \star_\hbar
        \big(
            f
            \star_\hbar
            g_+
        \big)
    \end{equation}
    holds for all $f \in \mathcal{H}(\D^2)$, $g_- \in \mathcal{H}(\Omega_-)$ and $g_+ \in \mathcal{H}(\Omega_+)$. Consequently, $\mathcal{H}(\D^2)$ is a continuous $\mathcal{H}(\Omega_-)$--$\mathcal{H}(\Omega_+)$--bimodule.
\end{proposition}
\begin{proof}
	Let $\hbar \in \mathscr{D}$, $f \in \mathcal{H}(\D^2)$, $g \in \mathcal{H}(\Omega_+)$ and $r \in (0,1)$. We begin by establishing the absolute convergence of the series \eqref{eq:StarProduct} in $\mathcal{H}(\D^2)$.
    Let $\alpha > 0$ such that \eqref{eq:PochhammerEstimate} holds.  We write
    \begin{equation*}
		\mathbb{B}_{r}
		\coloneqq
		\big\{
		      (z,w)
		      \in
			\C^2
			\colon
			\abs{z} \le r,
			\abs{w} \le r
		\big\}
	\end{equation*}
	for the closed bi--disk with radius $(r,r) \in [0,\infty)^2$. Note that for $0 < R_1 < \tfrac{1-r}{1+r} < 1$ we have
	\begin{equation*}
	   \abs[\bigg]
	   {\frac{z+u}{1+uw}}
	   \le
	   \frac{r+R_1}{1-rR_1}
					<
					\frac{r+\tfrac{1-r}{1+r}}{1-r \tfrac{1-r}{1+r}}
					=
					\frac{r(1+r)+1-r}{1+r-r(1-r)}
					=
					\frac{r^2+1}{1 + r^2}
					=
					1
				\end{equation*}
	for all $(z,w) \in \mathbb{B}_r$ and $\abs{u} \le R_1$. As in the proof of Lemma~\ref{lem:PMCauchy}, the Cauchy estimates imply
				\begin{equation*}
					{\norm[\big]
					{D^n_w f}}_{\mathbb{B}_r}
					\le
					\frac{n!}{R_1^n}
					\cdot
					{\norm{f}}_{K_1}
					\, ,
					\qquad
					n \in \N_0
				\end{equation*}
	with $K_1 \coloneqq \big\{\big(\tfrac{z+u}{1+uw},w\big) \colon (z,w) \in \mathbb{B}_r, \abs{u} = R_1\big\}$. Applying now Lemma~\ref{lem:PMCauchy} with $R_2 \coloneqq \tfrac{2\alpha}{R_1}$ gives a corresponding compact set $K_2 \subseteq \Omega_+$ such that
				\begin{equation*}
					{\norm[\big]
					{D^n_w g}}_{\mathbb{B}_r}
					\le
					\frac{n!}{R_2^n}
					\cdot
					{\norm{g}}_{K_2}
					\, ,
					\qquad
					n \in \N_0.
				\end{equation*}
	Putting everything together yields
	\begin{equation*}
		\sum_{n=0}^\infty
        \frac{1}{n!}
        \frac{1}{\abs{(-1/\hbar)_{n \downarrow}}}
        {\norm[\big]
        {B_n(g,f)}}_{\mathbb{B}_r}
		\le
		{\norm{f}}_{K_1}
		\cdot
		{\norm{g}}_{K_2}
		\cdot
		\sum_{n=0}^\infty
		\frac{\alpha^n}{R_1^n R_2^n}
        =
        {\norm{f}}_{K_1}
		\cdot
		{\norm{g}}_{K_2}
        \cdot
        2 \, .
	\end{equation*}
    That is, \eqref{eq:StarProduct} converges absolutely in $\mathcal{H}(\D^2)$ to some $f \star_\hbar g \in \mathcal{H}(\D^2)$. The second mapping in \eqref{eq:ModuleOne} may be treated analogously. The resulting mappings \eqref{eq:ModuleOne}
    inherit bilinearity from $\star_\hbar \colon \mathcal{H}(\Omega) \times \mathcal{H}(\Omega) \longrightarrow \mathcal{H}(\Omega)$ and \eqref{eq:BimoduleCompatibility} follows from its associativity. Finally, our considerations imply
    \begin{equation*}
        {\norm[\big]
        {f \star_\hbar g}}_{\mathbb{B}_r}
        \le
        2
        {\norm{f}}_{K_1}
		\cdot
		{\norm{g}}_{K_2}
        \, ,
    \end{equation*}
    proving also the continuity of both maps in \eqref{eq:ModuleOne}.
\end{proof}

    By the same arguments as in Theorem~\ref{thm:StarProduct}, the dependence on $\hbar$ is holomorphic.
    \begin{corollary}
        \label{cor:ModuleHolomorphic}
        Let $f \in \mathcal{H}(\D^2)$, $g_+ \in \mathcal{H}(\Omega_+)$ and $g_- \in \mathcal{H}(\Omega_-)$. Then the mappings
		\begin{equation*}
			\mathscr{D}
			\ni
			\hbar
			\mapsto
			f \star_\hbar g_+
			\in
			\mathcal{H}(\D^2)
			\quad \text{and} \quad
			\mathscr{D}
			\ni
			\hbar
			\mapsto
			g_- \star_\hbar f
			\in
		      \mathcal{H}(\D^2)
	   \end{equation*}
		are holomorphic.
        \end{corollary}

	As $\Omega_\pm \subseteq \Omega$, it makes sense to ask for a similar result there. Taking another look at the proof of Theorem~\ref{thm:StarProduct} and Lemma~\ref{lem:PMCauchy}, we have already shown the following.
	\begin{corollary}[Module structures II]
		\label{cor:StarProductModuleOmega}%
				Let $\hbar \in \mathscr{D}$. The bilinear mappings
				\begin{equation*}
					\star_\hbar
					\colon
					\mathcal{H}(\Omega)
					\times
					\mathcal{H}(\Omega_+)
					\longrightarrow
					\mathcal{H}(\Omega_+)
					\quad \text{and} \quad
					\star_\hbar
					\colon
					\mathcal{H}(\Omega_-)
					\times
					\mathcal{H}(\Omega)
					\longrightarrow
					\mathcal{H}(\Omega_-)
				\end{equation*}
				are well--defined and continuous. Hence, $\mathcal{H}(\Omega_+)$ is a continuous $\mathcal{H}(\Omega)$--rightmodule and $\mathcal{H}(\Omega_+)$ a continuous $\mathcal{H}(\Omega)$--leftmodule.
			\end{corollary}

			The dependence on the deformation parameter is holomorphic in analogy with Corollary~\ref{cor:ModuleHolomorphic}. Finally, we remark that given $f \in \mathcal{H}(\Omega_-)$ and $g \in \mathcal{H}(\Omega_+)$, also the product $$f \star_\hbar g  \in \mathcal{H}(\Omega \cap \C^2)$$ is well--defined. However, this process drops all the information near infinity.

			\section*{Acknowledgements}
                        The authors thank Daniela Kraus, Sebastian Schlei{\ss}inger, Matthias Sch\"otz and Stefan Waldmann for countless helpful and inspiring discussions. The research of A.~Moucha is partially funded by the Alexander von Humboldt foundation. Part of this work was done while T.~Sugawa was visiting the University of W\"urzburg partially supported by the Alexander von Humboldt foundation. The authors are indebted to the anonymous referee for a very careful reading of the original manuscript and her/his many thoughtful suggestions which considerably improved the exposition.

		\end{document}